\documentclass[12pt]{article}

\newif\ifsaneUnicode

\ifsaneUnicode 

	\usepackage{unicode-math}
	
	\newcommand{\−}{\textsuperscript{−−}}  

\else 

	\usepackage[utf8]{inputenc}
	\newcommand{\pred}{^{--}}
	\usepackage{amssymb}  

\fi 

\usepackage{cancel}
\usepackage{comment}
\usepackage{microtype}
\usepackage{amsmath}
\usepackage{amsthm}
\usepackage{censor}
\StopCensoring

\title{A Closer Look at the Russell Paradox}
\author{\censor{Flash Sheridan \\ \textit{\footnotesize{http://flash-sheridan.name  flash@pobox.com}}} }
  
\date{\footnotesize{\censor{20 December 2005, revised and extended 14 September 2021 and} 27 May 2024} \\ \normalsize{\textit{To appear in Logique et Analyse}}}

\newcommand{\bI}{\begin{it}}
\newcommand{\eI}{\/\end{it}}

\newcommand{\Def}[1]{\textbf{#1}\index{#1}}

\newcommand{\+}{\textsuperscript{++}}
\newcommand{\rep}{\text{-}rep}
\newcommand{\MyIndex}[1]{#1\index{#1}}

\theoremstyle{definition}
	\newtheorem{defn}{Definition}[section]
	\newtheorem{axiom}{Axiom}[section]
	\newtheorem{theorem}{Theorem}[section]
		\newtheorem{lemma}[theorem]{Lemma}
		\newtheorem{corollary}[theorem]{Corollary}
		\newtheorem{remark}[theorem]{Remark}

\begin{document}
\maketitle
\AtEndDocument{\medskip \tiny{© \censor{Flash Sheridan 1980, 2005, 2021–4}}}

\begin{abstract}
	 Two types of approximation to the paradoxical Russell Set are presented, one approximating it from below, one from above.  It is shown that any lower approximation gives rise to a better approximation containing it, and that any upper approximation contains a distinct better approximation.  The Russell Paradox is then seen to be the claim that two of these processes of better approximations stop, and at the same set.  This suggests that the unrestricted Axiom of Comprehension is, not a coherent intuition worthy of rescue from a mysterious paradox, but simply wishful thinking, a confusion of sets as extensional objects with classes defined by a property.
\end{abstract}

\ifsaneUnicode 

\section*{Introduction}
	 I present a fresh look at the Russell Paradox, using only trivial set existence axioms, with approximations to the hypothetical Russell Set of all sets which do not belong to themselves.  This will reveal that an approximation which contains too little leads to a better approximation which still contains too little.  Similarly, an approximation which contains too much will give rise to a regress of successively less bad approximations.  Each such series of approximations turns out to be indefinitely extensible, so an intuition that these  processes terminate will not withstand scrutiny.  I will show, for each of the major types of philosophical intuition about set existence, how that intuition fails to cohere with a supposed intuition implying the existence of the Russell Set.  
	 
	 The Russell Paradox will thus be seen to be a simple mistaken claim about set existence; so the unrestricted Axiom of Comprehension is a mistake to be abandoned, not an important intuition to be rescued, despite the suggestions of, among others, Kurt Gödel and Alonzo Church to the contrary.\footnote{See the quotations in §\ref{The Platonist Conception} “The Platonist Conception” and §\ref{Ill-Founded Iterative Conception} “The Ill-Founded Iterative Conception,” respectively.}  This will suggest that attempts to reconcile Comprehension with set theory Platonist enough to have a \MyIndex{universal set}, such as Quine’s New Foundations \cite{Quine 1937}, are philosophically misguided. 

	 The alternative intuition proposed below is a fully Platonist conception of sets as extensional objects without reliance on the temporal metaphor of construction.  A Platonist may observe that sets are closed under certain operations, though that is different from creating new ones; 
	 but this article requires axioms only for two trivial operations.  Church’s “Set Theory with a Universal Set” \cite{Church 1974a} and Oberschelp’s “Set Theory over Classes” \cite{Oberschelp 1973} are encouraging indications that such a theory might be consistent.

 \section{Shared Intuitions} \label{Shared Intuitions}
Formalizing mathematical intuitions, especially multiple conflicting intuitions on the nature of set existence, is obviously perilous; but two beliefs seem shared across some broad groups of intuitions about sets.  

\medskip
	  	Before presenting those, for historical and motivational context, I note the Unrestricted Axiom Schema of Comprehension, which I do \textit{not} assert, and which leads immediately to the Russell Paradox (\cite{Frege 1893}, “Afterword”).  This Schema is the simplification and restriction to \MyIndex{first-order} set theory of Frege’s more general Axiom V (\cite{Frege 1893}, §20 \& §47).  Conventional Zermelo-Fraenkel set theory’s Axiom Schema of Separation restricts the set whose existence is posited by the axiom to being a subset of an existing set; Quine’s New Foundations \cite{Quine 1937} has a Comprehension Schema with a syntactic restriction on the formula \(ϕ\).

 \begin{axiom}{Unrestricted Axiom Schema of Comprehension} \textit{(not asserted)} 
 
 For an arbitrary predicate \(ϕ\) with \(y\) not free in \(ϕ\), \[∃y∀x.\ x∈y ⇔ ϕ(x)\]
 \end{axiom}
 
 \begin{lemma} The Russell Paradox: The Unrestricted Axiom Schema of Comprehension leads immediately to a contradiction. \end{lemma}
 \begin{proof}
 	Replacing “\(ϕ\)” with “\(x∉x\)”, we obtain \(∃y∀x.\ x∈y ⇔ x∉x\).  (This \(y\), if it actually existed, would be the \MyIndex{Russell Set}: \(\{x \mid x∉x\}\).)  Instantiating \(y\) and substituting \(y\) for \(x\), we obtain \(y∈y ⇔ y∉y\).
 \end{proof}

 \smallskip

\medskip
One shared intuition, I claim, is extensionality, i.e., that two sets with the same members are identical, perhaps excepting the \MyIndex{empty set} or \MyIndex{urelements}.  Individuation by its members is arguably part of the sense of the concept of set, and extensionality is shared by all the intuitions of set which I will consider below.  
 An Axiom of \MyIndex{Extensionality} will not however be needed, only the trivial observation that if something is a member of one set but not another, then the two sets are distinct.  
 Some amount of extensionality will be implied by the uniqueness requirements in the two axioms of \MyIndex{Incomprehensive Set Theory}, below.
 
 The second shared intuition is that if two (informal) collections differ only by a single element, then if one is a set, so is the other.  
 This may seem trivial, but is violated by Skala’s Set Theory \cite{Skala 1974}, which is essential to that theory’s avoidance of the Russell Paradox (implicit in \cite{Kuhnrich 1980}).  I claim that this renders that theory merely a curiosity rather than a serious attempt to capture intuitions about set existence.  
 I do not need the second shared intuition in full generality, but instead take as axioms two special cases, where the single differing element is one of the two sets.

A third point, which is a truism in logical endeavors, but perhaps merely a desideratum for intuition, is that a responsible intuition of a generalization must allow intuiting one of its instances.  How to react to a counterexample to an appealing general intuition is not straightforward, and Frege, Gödel, and Church had different reactions:  see the quotations in §\ref{Restrictive Repair Attempts} “Restrictive Repair Attempts,” §\ref{The Platonist Conception} “The Platonist Conception,” and §\ref{Ill-Founded Iterative Conception} “The Ill-Founded Iterative Conception,” respectively.

 \section{Language and Axioms}
 We work in classical \MyIndex{first order} logic with equality, plus a primitive predicate symbol \textit{“set(…)”}.  Little use will be made of this symbol; it occurs only in the axioms to avoid making assumptions about the existence or non-existence of \MyIndex{urelements}.  
 
 \medskip
 \Def{Incomprehensive Set Theory} will consist of only two axioms,  the Existence of Successor and of Predecessor, which we will use for the results to follow.  
	  	One purpose of these minimal axioms is meta-theoretic (or, rather, pre-axiomatic), as a warning of a shortcoming any plausible axiomatization of intuitions about sets must avoid.  
	  	
 	 Any assumptions about set existence, even that any sets at all exist, will be noted explicitly.
 	 What follows will be trivial unless a Lower Russellian Set or an Upper Russellian Set (defined below) exists, and the main interest is when an empty set and a \MyIndex{universal set} exist.  
 	 The ultimate goal is casting light on set theories with non-trivial existence assumptions, but for generality I keep my theory separate from those additional assumptions.

	 \medskip

	 Informally the first axiom states that, for any object \(x\), its successor \(x ∪ \{x\}\) exists and is a unique set;\footnote{I follow \cite{Church 1974a} p.\ 300 in defining this for arbitrary objects, not just numbers.} it will be abbreviated as “\(x\+\)”\index{++}.  Formally,  

		 \begin{axiom}{Existence of Successor} \[ ∀x ∃!y.\  set(y) ∧ ∀z. \ z∈y ⇔ z∈x ∨ z=x \] \end{axiom} 

 \smallskip
	 Informally, for any object \(x\), \(x ∖ \{x\}\) exists and is a unique set; it will be abbreviated as “\(x\−\)\index{-{}-}”.  Formally,  

		 \begin{axiom}{Existence of Predecessor} \[ ∀x ∃!y.\  set(y) ∧∀z.\ z∈y ⇔ z∈x ∧ z≠x \]  \end{axiom}  

	 Despite the notation, no claim is made about whether \(x = x\+\) or \(x = x\−\), or whether the two operations are inverses of each other.  
	 The uniqueness requirement in each axiom is a notational prerequisite to use the two operator symbols as functions.  The sethood requirements are needed to avoid unwanted commitments to the existence or non-existence of urelements.  Without them, for instance, the Predecessor Axiom would rule out the existence of urelements if a \MyIndex{Quine atom} (a set whose only member is itself) existed.
	 
 \begin{remark} Degeneracy \label{degeneracy} \end{remark} Trivially, something is coextensive with its successor iff it is self-membered; it is coextensive with its predecessor iff it is not self-membered. 

	 Thus the non-degenerate case of either \(x\−\) or \(x\+\) is coextensive with \(x Δ \{x\}\), where “\(Δ\)” denotes symmetric difference.  
	 This makes it easy to modify constructions like the interpretation in Church’s Set Theory with a Universal Set \cite{Church 1974a}, or \cite{Sheridan 2016}, to satisfy both the Axiom of Successor and of Predecessor.

 \section{Definitions}  
 \begin{defn}
	 Call something a \Def{Lower Russellian Set} (or just a \Def{lower}, for short) when all its members are non-self-membered.  In symbols: 
	 \[ \Def{lower}(x) ≝ ∀z∈x.\ z∉z \]
\end{defn} 
\begin{defn}
	 Call something an \Def{Upper Russellian Set} (or just an \Def{upper}, for short) when it contains every non-self-membered set.  In symbols: 
	 \[ \Def{upper}(x) ≝ ∀z.\ z∉z ⇒ z∈x \]
 \end{defn}
 
 \begin{defn}
	 Call something a \Def{Strictly Russellian Set} if it is both a Lower Russellian Set and an Upper Russellian Set, i.e.,
 	 \[ \Def{strictly-russellian}(x) ≝ ∀z.\ z∈x ⇔ z∉z \]
\end{defn}

Informally, \(x\) and \(y\) constitute an \Def{ascending link} iff they are distinct and \(x\) belongs to \(y\), and a \Def{descending link} iff they are distinct and \(y\) belongs to \(x\).  
Two definition schemata for expository convenience: for an arbitrary predicate \(ϕ\), \(x\) and \(y\) constitute a \Def{\(ϕ\) ascending link} iff they constitute an ascending link and both have the property \(ϕ\), and analogously for a \Def{\(ϕ\) descending link}.  Formally:
\begin{defn}
  \[ \Def{ascending link}(x, y) ≝ x∈y ∧ x≠y \]
\end{defn}
\begin{defn}
  \[ \Def{descending link}(x, y) ≝ y∈x ∧ x≠y \]
\end{defn}
\begin{defn}
  \[ \Def{\(\phi\) ascending link}(x, y) ≝ x∈y ∧ x≠y ∧ ϕ(x) ∧ ϕ(y) \]  
\end{defn}
\begin{defn}
  \[ \Def{\(\phi\) descending link}(x, y) ≝ y∈x ∧ x≠y  ∧ ϕ(x) ∧ ϕ(y)\]
\end{defn}

\pagebreak[3]
\section{Trivial Lemmata}
	 \begin{lemma} \label{1} A lower is not a member of itself. \end{lemma}
\begin{proof}\nopagebreak[4]
			 Informally, if it were, it couldn’t be.  

		 Formally, assume \(x\) is a lower, and \(x∈x\).  Since \(∀z∈x.\ z∉z\) by definition, \(x∉x\).  
 \end{proof}

	 \begin{lemma} \label{2} An upper is a member of itself. \end{lemma}  

\begin{proof}
			 Informally, if it weren’t, it would have to be.  

		 Formally, assume \(x\) is an upper, and \(x∉x\).  By definition, \(∀z.\ z∉z ⇒ z∈x\), so, substituting \(x\) for \(z\), \(x∈x\).  
\end{proof}

	 \begin{corollary} \label{3} No set is both an upper and a lower. \end{corollary}  
	 This amounts to a restatement of the Russell Paradox:  the existence of a strictly Russellian Set leads immediately to a contradiction.  A more illuminating restatement is below, Theorem \ref{Russell Paradox Restated}.
	 
	 \begin{corollary} \label{stoppage} A lower is coextensive with its predecessor, and an upper is coextensive with its successor, by Remark \ref{degeneracy}. \end{corollary}  

	 \begin{lemma} \label{pred self} \(x∉x\−\). \end{lemma}  
\begin{proof}
	By the definition of predecessor, \(∀z.\ z∈x\− ⇔ z∈x ∧ z≠x\), so substituting \(x\) for \(z\), \(x∈x\− ⇔ x∈x ∧ x≠x\).  Thus \(x∉x\−\).
\end{proof}

	 \begin{lemma} \label{4} Similarly, \(x∈x\+\). \end{lemma}

\pagebreak[3]
 \section{Lemmata} 
	 \begin{lemma} \label{A} If \(x\) is a lower, then \(x\+ ≠ x\). \end{lemma}  \nopagebreak[4]

\begin{proof}
			 Assume not.  So \(x\+ = x\), and \(x∈x\+\) by Lemma \ref{4}.  So \(x∈x\), and \(x\) is not a lower.  
\end{proof}

	 \begin{lemma} \label{B} If \(x\) is a lower, then \(x\+\) is a lower. \end{lemma}   

\begin{proof}
			 Let \(z\) be a member of \(x\+\).  So \(z∈x ∨ z=x\).  But \(x\) is a lower, so in the first case, \(z\) is non-self-membered.  In the second case, since \(x\) is a lower, it is non-self-membered by Lemma \ref{1}.  
\end{proof}

	 \begin{lemma} \label{C} If \(x\) is an upper, then \(x\−≠x\). \end{lemma}  

\begin{proof}
			 By Lemma \ref{pred self}, \(x∉x\−\).  But \(x∈x\) by Lemma \ref{2}, so \(x\) and \(x\−\) are distinct.  
\end{proof}

	 \begin{lemma} \label{D} If \(x\) is an upper, then \(x\−\) is an upper. \end{lemma}  

\begin{proof}
			 Let \(z\) be non-self-membered; show that it is in \(x\−\).  By hypothesis \(z\) is in \(x\).  By the definition of predecessor, \(z∈x\− ⇔ z∈x ∧ z≠x\); so \(z∈x\− ⇔ z≠x\).  But \(x\) is an upper, hence \(x∈x\) by Lemma \ref{2}.  So \(z\), which is non-self-membered, is distinct from \(x\); thus \(z∈x\−\).  
\end{proof}

	 \begin{lemma} \label{E} If \(x\) is an upper, then \(x\−∈x\). \end{lemma}  

\begin{proof}
			 By Lemma \ref{D}, \(x\−\) is an upper, so by Lemma \ref{2}, \(x\−∈x\−\).\footnote{Thomas Forster has queried in personal communication whether there are cases, except for pathologies like \(q = \{q\}\), where \(x∈x\) but \(x\−∉x\−\).  I sketch in the Appendix a modification of the construction in \cite{Sheridan 2016} with such a set, satisfying \(N = U ∖ \{N\−\}\), where \(U\) is the \MyIndex{Universal Set}.\label{footnote_Forster}}
		 By the definition of predecessor, \(x\−\) is a subset of \(x\), so \(x\−∈x\).
\end{proof}

 	Thus, informally, starting with any lower (for instance the \MyIndex{empty set}), there is a distinct lower containing it, and so on.
	Starting with any upper (for instance the \MyIndex{Universal Set}), there is a distinct upper contained in it, and so on.\footnote{Note that the minimal axioms of \MyIndex{Incomprehensive Set Theory} do not provide enough machinery to formalize the above “and so on” with sets, so they remain intuitive indefinitely-extensible processes, rather than what might properly be called infinite ascending or descending chains.  In particular, the axioms do not address taking limits, though note that hypothetically, if the union of a collection of lowers exists, or the intersection of a collection of uppers, then it is also a lower or upper respectively.  \cite{Church 1974a} p.\ 304 notes that it is an open question whether there is a plausible theory with a set with an infinite descending chain but without a membership cycle.}
	More formally:

\begin{theorem} \label{Main Result}
Main Result 
\begin{enumerate}
  \item No lower link intersects any upper link.
  \item 	For any lower \(x\), \(x\) and \(x\+\) constitute a lower ascending link.
  \nopagebreak[4]
  \item For any upper \(x\), \(x\) and \(x\−\) constitute an upper descending link.
\end{enumerate}
 \end{theorem}  
	 
\begin{proof}
In three parts.
\begin{enumerate}
   \item Corollary \ref{3} states that nothing is both an upper and a lower. 
   \item By Lemma \ref{A}, a lower’s successor is distinct, by Lemma \ref{B} is also a lower, and by Lemma \ref{4} contains it.
   \item By Lemma \ref{C}, an upper’s predecessor is distinct, by Lemma \ref{D} is also an upper, and by Lemma \ref{E} is contained by it.
\end{enumerate}
\end{proof}

 \section{The Russell Paradox Restated}
	 The Russell Paradox, and hence the unrestricted Axiom of Comprehension, imply the claim that one such ascending process, and one such descending process, both \textit{stop,} and at the same set, which is both an upper and a lower.  Stated formally, arranged to display the contrast with the preceding theorem:
	 
	 \begin{theorem} Russell Paradox Restated \label{Russell Paradox Restated} \\
	 Let \(R\) be a strictly Russellian set.
\begin{enumerate}
  \item \(R\) is both an upper and a lower.
  \item 	\(R\+ = R\) 
  \nopagebreak[4]
  \item \(R\− = R\)
\end{enumerate}
	 
	  \end{theorem}
	 
	\begin{proof}
\begin{enumerate}
  \item By the definition of strictly Russellian.
  \item By \ref{stoppage}, since \(R\) is an upper.
  \item Also by \ref{stoppage}, since \(R\) is a lower.
\end{enumerate}
\end{proof}

   \section{Philosophical Interpretations}\label{Philosophical Interpretations}

	 Leaving the realm of proof, and entering that of intuition, I claim that the above restatement of the \MyIndex{Russell Paradox} shows that the \MyIndex{Axiom of Comprehension} does not count as a worthwhile intuition deserving of rescue or repair.\footnote{See \textit{contra,} for example, the quotations below from Gödel in §\ref{The Platonist Conception} “The Platonist Conception” and Church in §\ref{Ill-Founded Iterative Conception} “The Ill-Founded Iterative Conception.”}  
	 On close examination, the conflation of predicates and sets (discussed below; “Begriffsumfange oder Klassen” in Frege’s terminology, “\MyIndex{extensions of concepts} or \MyIndex{classes}”) is not essentially correct, with minor blurring, as Gödel would like and Frege feels forced to assume.\footnote{See the quotation below in §\ref{Restrictive Repair Attempts} “Restrictive Repair Attempts.”}  A better visual analogy will be an optical illusion, where the initial impression turns out to be without foundation, and must be adjusted in the light of reality.  
	 
	 As discussed in §\ref{Shared Intuitions} “Shared Intuitions”, a responsible intuition of a generalization must allow intuiting one of its instances, in this case the existence of the \MyIndex{Russell Set}.  
	 This intuition depends on one’s basic intuition about the nature of set existence, which seems to come in three relevant varieties, which I will call the iterative conception, the lasso, and the Platonist.\footnote{I do not claim that my analysis is relevant to some very different conceptions: positive and related theories such as those of Malitz and Barwise \& Moss; ill-founded set theories without a universal set, such as Forti \& Honsell’s and Aczel’s; and the type theory of  \MyIndex{Whitehead} \& \MyIndex{Russell}’s \MyIndex{\textit{Principia Mathematica}}.}  For each of these conceptions, a supposed intuition of the existence of the Russell Set does not withstand scrutiny: the supposed Comprehension intuition does not cohere with the conception.  A fourth category, which I discuss first, is restrictive repair attempts without a conspicuous intuition behind them, for instance as in Frege’s hurried addition (below) to his \textit{Grundgesetze} in reaction to Russell’s letter, or Quine’s New Foundations \cite{Quine 1937}.
	 
 \subsection{Restrictive Repair Attempts} \label{Restrictive Repair Attempts}
 To Frege’s credit, his reaction to Russell’s letter describing his paradox was acceptance rather than denial; he rushed into print, however, a hurried restriction (without a philosophical justification) to his Axiom V  (which in modern terminology is a \MyIndex{higher-order} generalization of the unrestricted \MyIndex{Axiom of Comprehension}) \cite{Frege 1893} “Afterword.”  This attempt is now generally called \MyIndex{Frege’s Way Out}.  

\begin{quotation}

Wie sollen wir uns hierzu stellen?  Sollen wir annehmen, das Gesetz vom ausgeschlossenen Dritten gelte von den Klassen nicht?  Oder sollen wir annehmen, es gebe Fälle, wo einem unanfechtbaren Begriffe keine Klasse entspreche, die sein Umfang wäre?……

Es bleibt also wohl nichts anderes übrig, als die \textbf{Begriffsumfange oder Klassen} als Gegenstände 
im eigentlichen und vollen Sinne 
dieses Wortes anzuerkennen, zugleich
aber einzuräumen, dass die bisherige
Auffassung der Worte „Umfang eines 
Begriffes" einer Berichtigung bedarf.

\textit{What attitude should we take to this?  Should we assume the law of excluded middle fails for classes?  Or should we assume that there are cases where to an incontestable concept no class corresponds as its extension?…}

\textit{So presumably nothing remains but to recognise \textbf{extensions of concepts or classes} as objects in the full and proper sense of the word, but to concede at the same time that the erstwhile understanding of the words “extension of a concept” requires correction. } \cite{Frege 1893} “Afterword,” pp.\ II.254–6, Ebert \& Rossberg translation.  Emphasis added.
\end{quotation}

Approximately four years later he apparently realized that his restricted axiom was inadequate for crucial proofs (\cite{Dummett 1991} p.\ 5).  He abandoned extensions of concepts completely and spent much of the rest of his career in apparent despair; he later started an attempt to found mathematics without using classes \cite{Frege 1924}.  In 1955, Quine showed that the restricted axiom proved that there was at most one object \cite{Quine 1955}.\footnote{Compare also Quine’s remarks on p.\ 156 with the definition of \MyIndex{Lower Russellian Set} and Lemma \ref{B}:  “we might require of a satisfactory reinterpretation that \(\hat{x}(ϕx)\) [the extension of the concept \(ϕ\)] continue to contain only objects \(x\) such that \(ϕx\).…  \(\hat{x}(ϕx)\) must not be a subclass of any further class which contains only objects \(x\) such that \(ϕx\)…  It may happen, for some open sentences, that every class whose members all fulfill the sentence is a subclass of a further class whose members all fulfill the sentence, so that there is no final class of the kind.”  See also \cite{Dummett 1991} p.\ 316.}  In 2009 Cook derived a contradiction based on the \MyIndex{Liar Paradox} from it \cite{Cook 2009}.

\smallskip
	 I claim that Frege indulged uncharacteristically here in what I would call the Fallacy of the Only Solution, and a flawed attempted correction (without a philosophical conception), to avoid facing squarely an instance where “to an incontestable concept no class corresponds as its extension.”  While modern set theories coherently treat “classes as objects in the full and proper sense of the word,” we now recognise that “extensions of concepts” and “classes” are not synonyms:  There exist both indefinable sets and extensionless concepts.

\bigskip
In 1937, Quine published his \MyIndex{New Foundations} set theory (\cite{Quine 1937}), which has a syntactically-restricted \MyIndex{Axiom of Comprehension}, also without a philosophical justification.  
In 1940 Quine presented a variant of this, \MyIndex{ML}, which Rosser showed suffered from the \MyIndex{Burali-Forti paradox} (revised edition of \cite{Quine 1951}, p.\ ix).

\bigskip  

From a Formalist perspective, inconsistencies simply remove all interest from Frege’s original system and his Way Out, and from Quine’s ML. 
New Foundations has fared much better for a Formalist; no contradiction has been proved in it, and considerable work has been done towards a relative consistency proof with \MyIndex{Zermelo-Frænkel set theory}.  It has one of Frege’s key desiderata, equinumerosity equivalence classes as sets, though its account of cardinality is counter-intuitive:  The cardinality of the universal set is different from that of the set of all singletons; although both are sets, the intuitively obvious mapping — the \MyIndex{singleton function} — cannot be a set (\cite{Holmes 1998}, p.\ 110, 131).  New Foundations also disproves the Axiom of Choice (\cite{Specker 1953}), though its variant NFU, which allows urelements, does not, and can be shown consistent in \MyIndex{Zermelo Set Theory} (\cite{Forster 1992} p.\ 67).  

\medskip
For Platonists, among others, \MyIndex{New Foundations} is not a philosophically satisfying solution to the Russell Paradox.  As Sir Michael Dummett wrote (\cite{Dummett 1991} p.\ 230), reflecting both his \MyIndex{Intuitionist} perspective and Frege’s \MyIndex{Logicism}:

\begin{quotation}
	Whatever mathematicians profess, mathematical theories conceived in a wholly formalist spirit are rare.  One such is Quine’s \MyIndex{New Foundations} system of set theory, devised with no model in mind, but on the basis of a hunch that a purely formal restriction on the comprehension axiom would block all contradictions.  The result is not a mathematical theory, but a formal system capable of serving as an \textit{object} of mathematical investigation:  without some conception of what we are talking about, we do not have a theory, because we do not have a subject-matter.
\end{quotation}

From a Platonist perspective such as Gödel’s (below), such concerns (along, perhaps, with mathematical elegance) are primary, and consistency (in sharp contradistinction to Formalism) while a necessary condition, is not alone sufficient.  Thus New Foundations would resemble \MyIndex{Frege’s Way Out} and \MyIndex{ML} in their early years, before their undesirable consequences (including, but not limited to, inconsistency) were discovered.

\subsection{The Iterative Conception}
	 The most influential variety of basic intuition is the iterative conception of set (e.g., \cite{Boolos 1971}); according to this intuition, sets are constructed in a manner understood via a temporal metaphor.  There are two subvarieties:  orthodox well-founded theories such as Zermelo-Fraenkel, and an ill-founded theory by Alonzo Church.
	 
	 For both subvarieties, the intuition implies the Axiom of the Existence of Successor quite directly, and requires the Axiom of the Existence of Predecessor slightly less directly:  Whatever metaphorical iteration led to a given set would also lead to its predecessor.  

\subsubsection{Well-Founded Iterative Conception}
	 Orthodox set theories such as Zermelo-Fraenkel\footnote{See \cite{Holmes 2017} for an argument that, since ZF allows sets to be constructed at an earlier level via quantification over sets constructed at a later level, the theory justified by the iterative conception is actually Zermelo Set Theory with \(Σ_{2}\) replacement.} satisfy the Axiom of Foundation and hence have no self-membered sets.  Thus all sets are lowers, and there are no uppers.  
	 
	 \smallskip
	 Comprehension’s implication of the existence of \MyIndex{Russell Set} then would amount to asserting the familiar \MyIndex{Mirimanoff} Paradox \cite{Mirimanoff 1917} of the set of all well-founded sets.  Having the conceived iteration not proceed even a single step beyond this set of all well-founded sets, would flout the iterative nature of the basic intuition.
	 
\subsubsection{Ill-Founded Iterative Conception}\label{Ill-Founded Iterative Conception}
Church’s published Set Theory with a Universal Set (\cite{Church 1974a}, p.\ 303) has both ill-founded sets and a 
 well-founded part which resembles that of the base model of Zermelo-Fraenkel with the \MyIndex{Axiom of Choice} used in his relative consistency proof. 
The axiom schemas of \textit{Aussonderung} (Separation), Replacement, and a sequence generalizing \MyIndex{Frege-Russell cardinals} are restricted to well-founded sets; so is the Axiom of Power Set, though this restriction (as well as some of the axioms and schemas) is removed in the related theory of \cite{Mitchell 1976}.
Church presents no intuition behind his theory here or in his unpublished notes and later related theories (\cite{Sheridan 2016} pp.\ 9–10).
Instead he suggests an historical heuristic:
 \begin{quotation}
  principle of specialized comprehension… we seek axioms which are special cases of the general comprehension axiom and which promise to maintain consistency while at the same time being adequate for a large variety of mathematical purposes.… may be roughly described by saying that we accept all of naive set theory except the comprehension axiom, and then assume as many special cases of the comprehension axiom as we dare  (p.\ 297).
\end{quotation}

His published article concludes by suggesting “the search for suitable axioms to be added to the basic axioms,”
with “one source of axioms to be studied for their consistency with the basic axioms” \cite{Hailperin 1944}’s finite axiomatization of \MyIndex{New Foundations} (pp.\ 307–8).
	\smallskip
	
	Assume that we intuit the existence of the \MyIndex{Universal Set} (or, indeed, any upper) while holding some sort of iterative conception of set existence.
	If we also intuit Comprehension, and hence the existence of \MyIndex{Russell Set}, 
	then one more iteration would yield its successor, which by the ascending link part of Theorem \ref{Main Result} is distinct.  The iterative process would also yield the \MyIndex{Russell Set}’s predecessor, which is likewise distinct.
	
	The iterative intuition requires continuing both these processes.  Calling a halt would be completely arbitrary; halting the successor iteration is analogous to claiming that the successor function on the natural numbers stops at some particular integer.  The predecessor iteration generates something new, given any upper: an indefinitely-extensible process of descending links.
	
	The analogous descending image I suggest for this is a non-standard model of Peano Arithmetic, visualizing the descending links as decreasing non-standard (i.e., externally non-finite) positive integers.  Visualize the ascending links as before: increasing non-negative standard (i.e., externally finite) integers.  Intuiting the existence of the Russell Set then is analogous to claiming that both processes stop, and at the same place.  This stopping point is then analogous to something which is both the last standard integer and the first non-standard integer, and which has neither a successor nor a predecessor.  Thus the consequences of the combined intuition do not cohere with the basic intuition.

	 \subsection{The Lasso and the Wand}
	 I address here the intuition presented in \cite{Forster 2008} on ill-founded set theories; he attributes the intuition of the lasso to \cite{Boolos 1971} on well-founded sets, who in turn credits Kripke for the original insight.  Boolos’ use of the lasso intuition, however, is for his iterative conception of well-founded sets, which is discussed above.  Forster’s employment of the lasso plus a wand for ill-founded sets is different; for generality and simplicity, I will use only the basic part of the intuition here. 
	 
	 This intuition also relies on a temporal metaphor, where some preexisting sets are selected by a lasso, and the wand creates two new sets: one containing the lassoed sets, and the other containing those preexisting sets not lassoed.  Note that the intuition must also determine whether each new set contains itself.  (Forster’s further intuitions supply these answers, and determine whether the new sets will contain each other, as well as later sets when this process is iterated; but that is not needed for our purposes.)  
	 
	 \smallskip
	 The crucial observation is simply that, whichever decision is made as to self-membership, and no matter what the lasso has collected, the decision avoids the Russell Paradox:  If a new set is self-membered, it is not a lower by Lemma \ref{1}; if it is not self-membered, it is not an upper by Lemma \ref{2}.  Either way, the intuition that some such set is strictly Russellian amounts to a refusal to make the required decision about membership in the newly-created set.

\bigskip	 
	\cite{Button 2023} contains a more general intuition, with wands satisfying various requirements, but not lassos, and with a modified temporal metaphor:  Sets are viewed, not as constructed at various levels, but merely indicated there.
	
	\smallskip
	The generality shifts the avoidance of paradox to requirements on particular wands.  The wand for equinumerosity (and for Church’s generalization), for example, satisfies these requirements (§11).  The obvious path to the \MyIndex{Russell Set}, the equivalence relation of being equally self-membered (i.e., \(x ∈ x ⟺ y ∈ y\)) is blocked by the requirements on a wand/set theory in his definition 5.1.
	 
\subsection{The Platonist Conception}\label{The Platonist Conception}
Mathematical Platonism is now less popular than the iterative conception, but it is worth recalling that it was part of the outlook of two of the greatest logicians in history, and it may still be relevant to understanding the \MyIndex{Russell Paradox}.
 \begin{quotation}
Die Zeit ist nur ein psychologisches Erforderniss zum Zählen, hat aber mit dem Begriffe der Zahl nichts zu thun.

\textit{Time is only a psychological necessity for numbering, it has nothing to do with the concept of number.}  

\cite{Frege 1884} §40 p.\ 53, J.L. Austin’s translation.
\end{quotation}

	 \begin{quotation}
 …it seems that the vicious circle principle in its first form [“definable only in terms of,” preceding page] applies only if the entities involved are constructed by ourselves.…

Classes and concepts may, however, also be conceived as real objects, … existing independently of our definitions and constructions.

It seems to me that the assumption of such objects is quite as legitimate as the assumption of physical bodies and there is quite as much reason to believe in their existence.

 \medskip
 
 …our logical intuitions would then remain correct up to certain minor corrections, i.e., they could then be considered to give an essentially correct, only somewhat “blurred,” picture of the real state of affairs. 
 
 \cite{Gödel 1944}, pp.\ 456, 466–7.
\end{quotation}

Frege’s original conception retains interest for Platonists even after the paradox.  It is of course a profound mystery how it achieved so much while its most significant axiom led to such a short contradiction.  But that achievement is enormous:  Frege provided an absolutely rigorous foundation for arithmetic, and was well on his way to providing a foundation for analysis.  Even \MyIndex{Whitehead} \& \MyIndex{Russell}’s \MyIndex{\textit{Principia Mathematica}} (according to \cite{Gödel 1944} p.\ 448) was substantially less rigorous.  Still less rigorous are typical logic articles today, not to mention proofs (more properly proof sketches) where there is serious disagreement over whether they actually constitute a proof.  Compare Frege’s boast:  
 \begin{quotation}
  Wenn etwa jemand etwas fehlerhaft finden sollte, muss er genau angeben können, wo der Fehler seiner Meinung nach steckt: in den Grundgesetzen, in den Definitionen, in den Regeln oder ihrer Anwendung an einer bestimmten Stelle. 
\end{quotation}
\begin{quotation}
   \textit{If anyone should believe that there is some fault, then he must be able to state precisely where, in his view, the error lies: with the basic laws, with the definitions, or with the rules or a specific application of them.}  \cite{Frege 1893}, p.\ VII, Ebert \& Rossberg translation
\end{quotation}

\medskip	 
	 By the Platonist conception, sets are not constructed in time, but exist independently of our constructions:  all we can do is point to them.  (“\MyIndex{Platonist}” is not ideal terminology; “\MyIndex{Pythagorean}” or “\MyIndex{atemporal}” might do as well.)  A set existence axiom merely points out a preexisting set with the required property; it no more creates a new set than we can create a new Boolean value.  
	 
	 \smallskip
	 Presumably these sets are closed under some interesting operations, but for this analysis, I assume only the usual trivial successor and predecessor axioms of \MyIndex{Incomprehensive Set Theory}.	 
	 In the case of the Russell Set, Theorem \ref{Main Result} makes it clear that no set we point to will be strictly Russellian:  Either the ascending or descending link part of the theorem will demonstrate the inadequacy of any plausible candidate, by pointing to a less bad approximation, which in turn is not strictly Russellian either.  
	 
	 By the conception of sets as existing independently of our constructions, the claim that there is a Russell Set is not, as Gödel would have it, an essentially correct intuition requiring minor correction.  
	 It is simply wishful thinking, analogous to wanting the phrase “sixth Platonic solid” to have a referent, while not actually intuiting which shape this might be.

\subsection{Conclusion}

I claim that Frege was right to abandon sets as extensions of concepts; a sufficiently deep and Platonist contemplation of this distinction, and of his rejection (quoted above) of the temporal metaphor, might have saved him from paradox and despair.  I also argue that this would not require abandoning sets completely, not even equivalence classes under equinumerosity as a formalization of cardinality.  
Rigorous definition by abstraction as the origin of cardinal numbers was the key to Frege’s philosophical anti-formalist doctrine that, as [Soames 2019] p.\ 97 summarizes, “\textit{Numbers are whatever they have to be in order to explain our knowledge of them.}  To discover what they are, we must give definitions of numbers that allow us to deduce what we pre-theoretically know,” 
which is a partial answer to the problem, still with no complete solution, of our knowledge of mathematical truths.  

It is also an alternative to arbitrariness in the implementation of cardinal numbers as initial Zermelo or von Neumann ordinals (\cite{Benacerraf 1965}, p.\ 281), and to the inconsistency of a “univocal role” for Frege cardinals (p.\ 284).\footnote{Frege’s 
definition of numbers in terms of one-one correspondence also pre-emptively rebuts Benacerraf’s claim (p.\ 290) that natural numbers have meaning only in the context of the sequence of all numerals.  A native speaker of a language without an indefinitely extensible numeric notation (\cite{Comrie 2022} p.\ 3), 
but who grasped the notions of bijection and pairs, might also be able to grasp Frege’s notion of the number two without first having to accept a sequence of all numbers.  

Recall Frege’s note (\cite{Frege 1884} §73 p.\ 85) that the etymology is deceptive here:  “\MyIndex{equinumerosity}” morphologically presupposes numerals, and Frege’s term “\MyIndex{gleichzahl}ig” morphologically presupposes numbers.  Austin’s translation “equal” avoids this difficulty but has its own problems.  “eindeutigen Zuordnung” (§62 p.\ 74) is the base concept, which Austin translates as “one-one correlation.”  The morphology may be helpful here:  in English, “one-point(-adjective)(-verb) to-arrange(-noun).”  Even the use of “one” increases the ontology unnecessarily; “each” would do as well.}  
It is an incomplete explanation of the otherwise unreasonable effectiveness of mathematics, since it does not cover higher mathematics, especially fields invented with no thought of practical applications which later acquired them.
\medskip

Considerable technical challenges remain, however, and even before Frege learned of the paradox, avoiding arbitrariness in the definition of real numbers (not to mention that of Julius \MyIndex{Cæsar}, \cite{Dummett 1991} pp.\ 157, 165, 280–1) had no obvious solution.\footnote{Given 
Frege’s suspicion of the unclarity of natural languages, the notorious Julius \MyIndex{Cæsar} problem would not have been relevant to his formalization program until it had reached formal definitions of names of biological entities.  While much of recent philosophy of mathematics and its notation is concerned with human language, this has not been a universal preoccupation.  (It would not be straightforward to read Frege’s two-dimensional concept writing aloud in any human language.  It is also worth recalling that Frege did not confine himself to a \MyIndex{first-order language}.)  

The first writing had no necessary connection with any particular natural language.  (Its inventors probably spoke primarily Sumerian, but the nature of the writing makes certainty impossible  \cite{Michalowski 2004}, pp.\ 20 \& 25.)  Given the current difficulties caused for artificial intelligences by reliance on predominantly one natural language (English), in the long run, the connection between human natural languages and logical writing may turn out to have been temporary.}

\cite{Church 1974a} has an axiom of the existence of a generalization of equinumerosity equivalence classes (including “relation number[s] in the sense of \textit{\MyIndex{Principia Mathematica}},” p.\ 303), though the axiom is restricted to well-founded sets, and considerably more work would need to be done to see whether its cardinalities are even more counterintuitive than Quine’s.  
Church apparently abandoned his unpublished attempts (\cite{Sheridan 2016} Appendix §C.2 pp.\ 99–100) at a unification between his set theories with a universal set and Quine’s set theory.  \censor{\cite{Sheridan 2016}} is a largely unsuccessful attempt to extend Church’s published theory while avoiding one of the difficulties of New Foundation’s cardinalities, the non-existence of the \MyIndex{singleton function}.
\cite{Oberschelp 1973} contains a promising approach for model-building and has the \MyIndex{singleton function} as a set, but both its philosophical underpinnings and the formalism of its main proof are far from obvious.  \cite{Burgess 2005} presents a quite different formalization of cardinal numbers, which might have reduced Frege’s despair over his foundational program.
\bigskip

	 That our intuitions are sometimes correct remains the most profound mystery of mathematics.  The failure of some of our intuitions can be considerably less profound, and needs to be accepted rather than resisted.	 In retrospect, it should have been no more surprising that not all predicates correspond to a set, than that not all sets correspond to a definable predicate.
	 Attempting to repair such a broken intuition would be the metaphorical equivalent, not of correcting blurred vision, but of searching for a smaller version of something after it turned out to be an optical illusion.

  \section{Appendix:  A Less Pathological Counterexample to a Self-Membership Query of Forster}
  
  In email about an earlier version of this paper, Thomas Forster raised the question of whether (in my notation): \(x∈x ⇒ x\− ∈ x\−\).  There are simple pathological counterexamples, such as a \MyIndex{Quine atom} (whose existence seems independent of \MyIndex{New Foundations}) \(q\) such that \(q = \{q\}\).  (Observe that \(q ∈ q\),  and \(q\− = q ∖ \{q\} = \{q\} ∖ \{q\} = ∅\), and \(∅ ∉  ∅\).)
  I am grateful to an anonymous referee for the observation that, in the context of \MyIndex{New Foundations}, a counterexample may be obtained with a \MyIndex{Rieger-Scott permutation}.

The generalization is provable in \MyIndex{Incomprehensive Set Theory} of any upper, by 4.4 \& 3.2, and trivially true of any lower by 3.1.  It holds in some simple cases in constructions like Church’s and \censor{mine}, e.g., the Universal Set U (where \(U∈U ∧ U\− ∈ U\−\)), and equivalences classes (or simple combinations thereof) such as the set of all singletons, which is not a member of itself.

That is however a limitation of those constructions rather than a general trait, according to the shared intuitions in §\ref{Shared Intuitions} of sets as extensional objects rather than only extensions of concepts.  It is a relic of the goal of this technique, to provide extensions for some broad concepts like equinumerosity equivalence classes; the constructions do not go very far beyond that goal.  By this intuition, whether one object belongs to a given set should be independent, in normal cases, of whether some other object belongs to it.  (A partial formulation of this intuition would be generalizations of the Axioms of Predecessor and Successor, to the existence of \(x ∖ \{y\}\) and of \(x ∪ \{y\}\), even when \(x ≠ y\).)

A non-pathological counterexample (which nonetheless exposes issues about individuation and identity), \(N\), with \(N = U ∖ \{N\−\}\), where \(U\) is the \MyIndex{Universal Set}, may be obtained via a small modification to the construction in \cite{Sheridan 2016}.  (This technique apparently would not work with the similar construction in \cite{Church 1974a}.)

This construction can be repeated, so that there is another such set with \(P = U ∖ \{P\−\}\), but this raises the question of whether it would accord with the intuition of sets as extensional objects to have two such sets.
Arguably this intuition requires that there be only one such object, but it is not clear what  general principle would justify this.  \MyIndex{Finsler}’s Axiom of \MyIndex{Strong Extensionality} might suffice, but may present difficulties \cite{Forster 1998}.
Such questions of identity, and axiomatising intuitions about the nature of sets as extensional objects, are promising areas for future work.

I here sketch the construction referenced in footnote \ref{footnote_Forster}, of a set satisfying \(N = U ∖ \{N\−\}\); consequently \(N∈N\) but \(N\−∉N\−\).  For non-readers of \cite{Sheridan 2016}, I begin with an informal summary of the construction technique.

\subsection[Brief Summary of  Membership Relation in \cite{Sheridan 2016}]{Brief Summary of  \(∈_3\) in \cite{Sheridan 2016}}

The \MyIndex{base theory} is \MyIndex{RZFU}, a theory similar to Zermelo-Frænkel with a global well-ordering, with the Axiom of Extensionality modified to permit urelements, and a bookkeeping axiom to keep track of them.  (Other differences from conventional ZFC with urelements need not concern us here,\footnote{E.g., the theory does not assume the Axiom of \MyIndex{Foundation}, but the axioms of \MyIndex{Replacement}, \MyIndex{Power Set}, and \MyIndex{Sum Set} are restricted to well-founded sets} \censor{and were in retrospect not worth the effort}.)
The construction below will use a relation similar to the main membership relation, \(∈_3\), in \cite{Sheridan 2016}, which I will call \(∈_4\).  
For the benefit of non-readers of that work, I will omit many minor details and sketch only the relevant parts of the construction.  A somewhat less brief summary is in \cite{Sheridan 2014}.

The \(∈_3\) construction is an interpretation over \censor{my} base theory (whose membership relationship is called \(∈_0\) for clarity), where \(∈_3\) extends \(∈_0\) when an urelement is on the right.
The point to the construction is to define \(∈_3\) in terms of \(∈_0\), and then prove the \Def{interpretation} of each axiom, i.e., its definitional expansion, including expansion of symbols defined in terms of \(∈_3\).  This then provides a purely syntactic\footnote{Despite Church’s use of the word “model,” e.g., page 305 of \cite{Church 1974a}.} proof of the consistency of the theory relative to the base theory:  Any proof of a contradiction from the axioms of interest can be converted into a proof of a contradiction in the base theory.

\subsubsection{The Sequence of Equivalence Relations}
Much of the complexity in both Church’s and  \censor{my} relative consistency proofs came from \censor{our} sequences of equivalence relations, which are not relevant to the current construction, and which I will ignore except for two trivial ones, the first and the last in the sequence.  \censor{My} sequence was from \(0\) to an arbitrary ordinal \(𝜇\), which was a generalization of Church’s choice of \(1\) through an arbitrary natural number \(m\).  For greater notational uniformity in that construction, \censor{I} supplemented Church’s relations with two trivial ones, \(≎_0\) and \(≎_𝜇\), which will be the only relations used here: \(≎_0\) is the trivial relation true of any two objects, and \(≎_𝜇\) is equality.  (The other ordinals between \(0\) and \(𝜇\), if any, will be elided here with “…”.) 

\censor{My} construction relied on a form of Global Choice to select a representative from each equivalence class; given the restriction to two trivial relations here, the user may think of \(𝜇\) as \(1\), and of the representative for \(x\)’s \(≎_𝜇\)-equivalence class (denoted “\(𝜇\rep(x)\)”) as \(x\) itself.  The representative of the sole \(≎_0\)-equivalence class (denoted “\(0\rep(x)\)”, for any \(x\)) may be thought of as \(∅\) for the sake of definiteness.

Sequences of length \(𝜇+1\) (which in conformity to Church’s terminology  \censor{I} called \(𝜇+1\)-tuples and surround with “\((…)\)” rather than “\(<…>\)”) are used to tag urelements.  These urelements represented new sets, such as the Universal Set, in the sense of \(∈_3\).  (\(∈_3\) coincides with \(∈_0\) when a non-urelement is on the right side.)  The tagging was via a bijective formula \(Υ''(x)\) (locally abbreviated, for convenience, “\(*\)”, with “\((…)\)” omitted; see \censor{§20.2 “Excess Urelements”} of \cite{Sheridan 2016} for details.)  The \(𝜇+1\)-tuple tagging an urelement was called its \textbf{Index}.

A simplifying intermediary \censor{I} added to Church’s definition was what \censor{I} called the “sprig of \(x\) for an urelement \(u\) with tag \(L\)”:  the collection of pairs \(<j, j\rep(x)>\) (for \(j ≤ 𝜇\)) such that \(j\rep(x)\) was in the \(j\)\textsuperscript{th} element of the tag \(L\) of the given urelement \(u\).

\subsubsection{\(∈_3\)}
The membership relation \(∈_3\) was extended for urelements (in the sense of \(∈_0\)), by stipulating for an urelement \(u\), that \(x ∈_3 u\) if the sprig of \(x\) for \(L\) had an odd number of members.  In the case we consider here, considering only ordinals \(0\) and \(𝜇\), the only possible odd number is one; i.e., if the sprig of \(x\) for \(L\) contains \(0\rep(x)\) or \(𝜇\rep(x)\) but not both.

Informally, in this type of construction, self-membership is mainly signalled only by the zeroth (complement) component, and can otherwise normally be affected only by the last (\(𝜇\)) component listing exceptions.  Components signalling Church’s or  \censor{my} equivalence relations don’t seem to provide self-membership; for instance, the Frege-Russell cardinal \(1\) (the set of all singletons), has more than one element.  (This type of interpretation doesn’t provide generalized Frege-Russell cardinals for big sets.) 

Thus, for instance, \(*(\{0\rep(0)\}, … \{\})\) will represent the \MyIndex{Universal Set}, U, in the sense of \(∈_3\):  The \(0\)\rep\  of anything is in the zeroth element of the sequence, and the other element(s) of the sequence are empty; thus the sprig of anything for tag \((\{0\rep(0)\}, … \{\})\) has one member, and one is odd.  Similarly, \(U\−\) would be represented by \(*(\{0\rep(0)\}, … \{𝜇\rep(U)\})\).

\subsection{\(∈_4\)}

I will here briefly sketch the small changes needed for a new membership relation \(∈_4\), which will provide the required counterexample.  The only change from \(∈_3\) in §20.3 of \cite{Sheridan 2016} is a tweak to the bijective formula \(Υ''\):  

Choose two distinct urelements \(M\) and \(N\).  Consider the following two indexes:

\[ n = (\{0\rep(0)\}, … \{𝜇\rep(M)\}) \]

\[ m = (\{0\rep(0)\}, … \{𝜇\rep(M), 𝜇\rep(N)\}) \]

We tinker with the bijective formula \(Υ''\) mapping Indexes to urelements.  Define \(Υ'''\) as \(Υ''\), except that we swap values so that it maps (if it does not already) \(n\) to \(N\) and \(m\) to \(M\), and (to avoid collisions) \(Υ''\text{-}inverse(N)\) to \(Υ''(n)\), and \(Υ''\text{-}inverse(M)\) to \(Υ''(m)\).  By my customary abuse of notation, I will again recycle the symbol “\(*\)” and locally abbreviate \(Υ'''(x)\) to \(*x\), without parentheses.

The definition of \(∈_4\) is then textually identical to that for \(∈_3\), given the new local definition of \(*\).

\subsection{Properties of M and N}

Thus we have 

\[ N = *(\{0\rep(0)\}, … \{𝜇\rep(M)\}) \]
\[ M = *(\{0\rep(0)\}, … \{𝜇\rep(M), 𝜇\rep(N)\}) \]

\begin{lemma} \(N\) is, in the sense of \(∈_4\), \(U ∖ \{M\}\) \end{lemma}

\begin{proof}
	By definition, \( x ∈_4 N ⇔ odd(sprig(N, x)) \).
There is only one \(0\rep\), so \(0\rep(x)\) is always in \{0\rep(0)\}.
So the sprig has an odd number of members unless \(𝜇\rep(x) ∈ \{𝜇\rep(M)\}\) is also true, i.e., iff \(x = M\).  So \(x ∈_4 N ⇔ x ≠ M\), as required.\end{proof}

\begin{corollary} \(N ∈_4 N\) \end{corollary}

\begin{lemma}  \(M\) is, in the sense of \(∈_4\), \(U ∖ \{M, N\}\) \end{lemma}

\begin{proof}
	Similarly, \(x ∈_4 M ⇔ odd(sprig(M, x))\),
which is true unless \(𝜇\rep(x) ∈ \{𝜇\rep(M), 𝜇\rep(N)\}\), iff
\(x = M ∨ x = N\).

So \(x ∈_4 M ⇔ x ≠ M ∧ x ≠ N\), as required.\end{proof}

\begin{corollary} \(M ∉_4 M\) \end{corollary}

\begin{corollary} \(M\) is, in the sense of \(∈_4\), \(N\−\) \end{corollary}

\begin{corollary} In the sense of \(∈_4\), \(N ∈_4 N ∧ N\− ∉_4 N\−\) \end{corollary}
\textit{Q.E.D.}

\section*{Acknowledgments}
	 \censor{This paper began in 2005 as a reconstruction of my essay on the \MyIndex{Russell Paradox} for \MyIndex{Angus Macintyre}’s Philosophy of Mathematics course in 1980, which provided the philosophical motivation for my work on Church’s Set Theory with a Universal Set \cite{Church 1974a}.  The reconstruction was heavily influenced by my unpublished article on singletons in Skala’s set theory \cite{Skala 1974}, which had been implicitly superseded by Kühnrich and Schultz’s earlier article \cite{Kuhnrich 1980}; Skala’s terminology is the origin of my terms “upper” and “lower.”  Various of Thomas Forster’s articles and corrections were also an influence, though the reconstruction was nearly complete before I read the most relevant of these, the November 2005 version of “The Iterative Conception of Set” \cite{Forster 2008}.}
	 
	 \censor{I am grateful to two anonymous referees for criticisms, suggestions, and demands for additional philosophical arguments, to my wife Olga Miroshnychenko for assistance while writing and rewriting, to Roger W. Janeway for discussions of the orginal version of this article and corrections and a metaphor for the final version, and to Murdoch J. Gabbay for suggestions which led to substantial improvements and additions.}

 \else 


\section*{Introduction}
	 I present a fresh look at the Russell Paradox, using only trivial set existence axioms, with approximations to the hypothetical Russell Set of all sets which do not belong to themselves.  This will reveal that an approximation which contains too little leads to a better approximation which still contains too little.  Similarly, an approximation which contains too much will give rise to a regress of successively less bad approximations.  Each such series of approximations turns out to be indefinitely extensible, so an intuition that these  processes terminate will not withstand scrutiny.  I will show, for each of the major types of philosophical intuition about set existence, how that intuition fails to cohere with a supposed intuition implying the existence of the Russell Set.  
	 
	 The Russell Paradox will thus be seen to be a simple mistaken claim about set existence; so the unrestricted Axiom of Comprehension is a mistake to be abandoned, not an important intuition to be rescued, despite the suggestions of, among others, Kurt Gödel and Alonzo Church to the contrary.\footnote{See the quotations in §\ref{The Platonist Conception} “The Platonist Conception” and §\ref{Ill-Founded Iterative Conception} “The Ill-Founded Iterative Conception,” respectively.}  This will suggest that attempts to reconcile Comprehension with set theory Platonist enough to have a \MyIndex{universal set}, such as Quine’s New Foundations \cite{Quine 1937}, are philosophically misguided. 

	 The alternative intuition proposed below is a fully Platonist conception of sets as extensional objects without reliance on the temporal metaphor of construction.  A Platonist may observe that sets are closed under certain operations, though that is different from creating new ones; 
	 but this article requires axioms only for two trivial operations.  Church’s “Set Theory with a Universal Set” \cite{Church 1974a} and Oberschelp’s “Set Theory over Classes” \cite{Oberschelp 1973} are encouraging indications that such a theory might be consistent.

 \section{Shared Intuitions} \label{Shared Intuitions}
Formalizing mathematical intuitions, especially multiple conflicting intuitions on the nature of set existence, is obviously perilous; but two beliefs seem shared across some broad groups of intuitions about sets.  

\medskip
	  	Before presenting those, for historical and motivational context, I note the Unrestricted Axiom Schema of Comprehension, which I do \textit{not} assert, and which leads immediately to the Russell Paradox (\cite{Frege 1893}, “Afterword”).  This Schema is the simplification and restriction to \MyIndex{first-order} set theory of Frege’s more general Axiom V (\cite{Frege 1893}, §20 \& §47).  Conventional Zermelo-Fraenkel set theory’s Axiom Schema of Separation restricts the set whose existence is posited by the axiom to being a subset of an existing set; Quine’s New Foundations \cite{Quine 1937} has a Comprehension Schema with a syntactic restriction on the formula \(\phi \).

 \begin{axiom}{Unrestricted Axiom Schema of Comprehension} \textit{(not asserted)} 
 
 For an arbitrary predicate \(\phi \) with \(y\) not free in \(\phi \), \[\exists y\forall x.\ x\in y \iff  \phi (x)\]
 \end{axiom}
 
 \begin{lemma} The Russell Paradox: The Unrestricted Axiom Schema of Comprehension leads immediately to a contradiction. \end{lemma}
 \begin{proof}
 	Replacing “\(\phi \)” with “\(x\notin x\)”, we obtain \(\exists y\forall x.\ x\in y \iff  x\notin x\).  (This \(y\), if it actually existed, would be the \MyIndex{Russell Set}: \(\{x \mid x\notin x\}\).)  Instantiating \(y\) and substituting \(y\) for \(x\), we obtain \(y\in y \iff  y\notin y\).
 \end{proof}

 \smallskip

\medskip
One shared intuition, I claim, is extensionality, i.e., that two sets with the same members are identical, perhaps excepting the \MyIndex{empty set} or \MyIndex{urelements}.  Individuation by its members is arguably part of the sense of the concept of set, and extensionality is shared by all the intuitions of set which I will consider below.  
 An Axiom of \MyIndex{Extensionality} will not however be needed, only the trivial observation that if something is a member of one set but not another, then the two sets are distinct.  
 Some amount of extensionality will be implied by the uniqueness requirements in the two axioms of \MyIndex{Incomprehensive Set Theory}, below.
 
 The second shared intuition is that if two (informal) collections differ only by a single element, then if one is a set, so is the other.  
 This may seem trivial, but is violated by Skala’s Set Theory \cite{Skala 1974}, which is essential to that theory’s avoidance of the Russell Paradox (implicit in \cite{Kuhnrich 1980}).  I claim that this renders that theory merely a curiosity rather than a serious attempt to capture intuitions about set existence.  
 I do not need the second shared intuition in full generality, but instead take as axioms two special cases, where the single differing element is one of the two sets.

A third point, which is a truism in logical endeavors, but perhaps merely a desideratum for intuition, is that a responsible intuition of a generalization must allow intuiting one of its instances.  How to react to a counterexample to an appealing general intuition is not straightforward, and Frege, Gödel, and Church had different reactions:  see the quotations in §\ref{Restrictive Repair Attempts} “Restrictive Repair Attempts,” §\ref{The Platonist Conception} “The Platonist Conception,” and §\ref{Ill-Founded Iterative Conception} “The Ill-Founded Iterative Conception,” respectively.

 \section{Language and Axioms}
 We work in classical \MyIndex{first order} logic with equality, plus a primitive predicate symbol \textit{“set(... )”}.  Little use will be made of this symbol; it occurs only in the axioms to avoid making assumptions about the existence or non-existence of \MyIndex{urelements}.  
 
 \medskip
 \Def{Incomprehensive Set Theory} will consist of only two axioms,  the Existence of Successor and of Predecessor, which we will use for the results to follow.  
	  	One purpose of these minimal axioms is meta-theoretic (or, rather, pre-axiomatic), as a warning of a shortcoming any plausible axiomatization of intuitions about sets must avoid.  
	  	
 	 Any assumptions about set existence, even that any sets at all exist, will be noted explicitly.
 	 What follows will be trivial unless a Lower Russellian Set or an Upper Russellian Set (defined below) exists, and the main interest is when an empty set and a \MyIndex{universal set} exist.  
 	 The ultimate goal is casting light on set theories with non-trivial existence assumptions, but for generality I keep my theory separate from those additional assumptions.

	 \medskip

	 Informally the first axiom states that, for any object \(x\), its successor \(x \cup  \{x\}\) exists and is a unique set;\footnote{I follow \cite{Church 1974a} p.\ 300 in defining this for arbitrary objects, not just numbers.} it will be abbreviated as “\(x\+\)”\index{++}.  Formally,  

		 \begin{axiom}{Existence of Successor} \[ \forall x \exists !y.\  set(y) \wedge  \forall z. \ z\in y \iff  z\in x \lor  z=x \] \end{axiom} 

 \smallskip
	 Informally, for any object \(x\), \(x \setminus  \{x\}\) exists and is a unique set; it will be abbreviated as “\(x\pred \)\index{-{}-}”.  Formally,  

		 \begin{axiom}{Existence of Predecessor} \[ \forall x \exists !y.\  set(y) \wedge \forall z.\ z\in y \iff  z\in x \wedge  z\neq x \]  \end{axiom}  

	 Despite the notation, no claim is made about whether \(x = x\+\) or \(x = x\pred \), or whether the two operations are inverses of each other.  
	 The uniqueness requirement in each axiom is a notational prerequisite to use the two operator symbols as functions.  The sethood requirements are needed to avoid unwanted commitments to the existence or non-existence of urelements.  Without them, for instance, the Predecessor Axiom would rule out the existence of urelements if a \MyIndex{Quine atom} (a set whose only member is itself) existed.
	 
 \begin{remark} Degeneracy \label{degeneracy} \end{remark} Trivially, something is coextensive with its successor iff it is self-membered; it is coextensive with its predecessor iff it is not self-membered. 

	 Thus the non-degenerate case of either \(x\pred \) or \(x\+\) is coextensive with \(x \Delta  \{x\}\), where “\(\Delta \)” denotes symmetric difference.  
	 This makes it easy to modify constructions like the interpretation in Church’s Set Theory with a Universal Set \cite{Church 1974a}, or \cite{Sheridan 2016}, to satisfy both the Axiom of Successor and of Predecessor.

 \section{Definitions}  
 \begin{defn}
	 Call something a \Def{Lower Russellian Set} (or just a \Def{lower}, for short) when all its members are non-self-membered.  In symbols: 
	 \[ \Def{lower}(x) \equiv_{df}  \forall z\in x.\ z\notin z \]
\end{defn} 
\begin{defn}
	 Call something an \Def{Upper Russellian Set} (or just an \Def{upper}, for short) when it contains every non-self-membered set.  In symbols: 
	 \[ \Def{upper}(x) \equiv_{df}  \forall z.\ z\notin z \implies  z\in x \]
 \end{defn}
 
 \begin{defn}
	 Call something a \Def{Strictly Russellian Set} if it is both a Lower Russellian Set and an Upper Russellian Set, i.e.,
 	 \[ \Def{strictly-russellian}(x) \equiv_{df}  \forall z.\ z\in x \iff  z\notin z \]
\end{defn}

Informally, \(x\) and \(y\) constitute an \Def{ascending link} iff they are distinct and \(x\) belongs to \(y\), and a \Def{descending link} iff they are distinct and \(y\) belongs to \(x\).  
Two definition schemata for expository convenience: for an arbitrary predicate \(\phi \), \(x\) and \(y\) constitute a \Def{\(\phi \) ascending link} iff they constitute an ascending link and both have the property \(\phi \), and analogously for a \Def{\(\phi \) descending link}.  Formally:
\begin{defn}
  \[ \Def{ascending link}(x, y) \equiv_{df}  x\in y \wedge  x\neq y \]
\end{defn}
\begin{defn}
  \[ \Def{descending link}(x, y) \equiv_{df}  y\in x \wedge  x\neq y \]
\end{defn}
\begin{defn}
  \[ \Def{\(\phi\) ascending link}(x, y) \equiv_{df}  x\in y \wedge  x\neq y \wedge  \phi (x) \wedge  \phi (y) \]  
\end{defn}
\begin{defn}
  \[ \Def{\(\phi\) descending link}(x, y) \equiv_{df}  y\in x \wedge  x\neq y  \wedge  \phi (x) \wedge  \phi (y)\]
\end{defn}

\pagebreak[3]
\section{Trivial Lemmata}
	 \begin{lemma} \label{1} A lower is not a member of itself. \end{lemma}
\begin{proof}\nopagebreak[4]
			 Informally, if it were, it couldn’t be.  

		 Formally, assume \(x\) is a lower, and \(x\in x\).  Since \(\forall z\in x.\ z\notin z\) by definition, \(x\notin x\).  
 \end{proof}

	 \begin{lemma} \label{2} An upper is a member of itself. \end{lemma}  

\begin{proof}
			 Informally, if it weren’t, it would have to be.  

		 Formally, assume \(x\) is an upper, and \(x\notin x\).  By definition, \(\forall z.\ z\notin z \implies  z\in x\), so, substituting \(x\) for \(z\), \(x\in x\).  
\end{proof}

	 \begin{corollary} \label{3} No set is both an upper and a lower. \end{corollary}  
	 This amounts to a restatement of the Russell Paradox:  the existence of a strictly Russellian Set leads immediately to a contradiction.  A more illuminating restatement is below, Theorem \ref{Russell Paradox Restated}.
	 
	 \begin{corollary} \label{stoppage} A lower is coextensive with its predecessor, and an upper is coextensive with its successor, by Remark \ref{degeneracy}. \end{corollary}  

	 \begin{lemma} \label{pred self} \(x\notin x\pred \). \end{lemma}  
\begin{proof}
	By the definition of predecessor, \(\forall z.\ z\in x\pred  \iff  z\in x \wedge  z\neq x\), so substituting \(x\) for \(z\), \(x\in x\pred  \iff  x\in x \wedge  x\neq x\).  Thus \(x\notin x\pred \).
\end{proof}

	 \begin{lemma} \label{4} Similarly, \(x\in x\+\). \end{lemma}

\pagebreak[3]
 \section{Lemmata} 
	 \begin{lemma} \label{A} If \(x\) is a lower, then \(x\+ \neq  x\). \end{lemma}  \nopagebreak[4]

\begin{proof}
			 Assume not.  So \(x\+ = x\), and \(x\in x\+\) by Lemma \ref{4}.  So \(x\in x\), and \(x\) is not a lower.  
\end{proof}

	 \begin{lemma} \label{B} If \(x\) is a lower, then \(x\+\) is a lower. \end{lemma}   

\begin{proof}
			 Let \(z\) be a member of \(x\+\).  So \(z\in x \lor  z=x\).  But \(x\) is a lower, so in the first case, \(z\) is non-self-membered.  In the second case, since \(x\) is a lower, it is non-self-membered by Lemma \ref{1}.  
\end{proof}

	 \begin{lemma} \label{C} If \(x\) is an upper, then \(x\pred \neq x\). \end{lemma}  

\begin{proof}
			 By Lemma \ref{pred self}, \(x\notin x\pred \).  But \(x\in x\) by Lemma \ref{2}, so \(x\) and \(x\pred \) are distinct.  
\end{proof}

	 \begin{lemma} \label{D} If \(x\) is an upper, then \(x\pred \) is an upper. \end{lemma}  

\begin{proof}
			 Let \(z\) be non-self-membered; show that it is in \(x\pred \).  By hypothesis \(z\) is in \(x\).  By the definition of predecessor, \(z\in x\pred  \iff  z\in x \wedge  z\neq x\); so \(z\in x\pred  \iff  z\neq x\).  But \(x\) is an upper, hence \(x\in x\) by Lemma \ref{2}.  So \(z\), which is non-self-membered, is distinct from \(x\); thus \(z\in x\pred \).  
\end{proof}

	 \begin{lemma} \label{E} If \(x\) is an upper, then \(x\pred \in x\). \end{lemma}  

\begin{proof}
			 By Lemma \ref{D}, \(x\pred \) is an upper, so by Lemma \ref{2}, \(x\pred \in x\pred \).\footnote{Thomas Forster has queried in personal communication whether there are cases, except for pathologies like \(q = \{q\}\), where \(x\in x\) but \(x\pred \notin x\pred \).  I sketch in the Appendix a modification of the construction in \cite{Sheridan 2016} with such a set, satisfying \(N = U \setminus  \{N\pred \}\), where \(U\) is the \MyIndex{Universal Set}.\label{footnote_Forster}}
		 By the definition of predecessor, \(x\pred \) is a subset of \(x\), so \(x\pred \in x\).
\end{proof}

 	Thus, informally, starting with any lower (for instance the \MyIndex{empty set}), there is a distinct lower containing it, and so on.
	Starting with any upper (for instance the \MyIndex{Universal Set}), there is a distinct upper contained in it, and so on.\footnote{Note that the minimal axioms of \MyIndex{Incomprehensive Set Theory} do not provide enough machinery to formalize the above “and so on” with sets, so they remain intuitive indefinitely-extensible processes, rather than what might properly be called infinite ascending or descending chains.  In particular, the axioms do not address taking limits, though note that hypothetically, if the union of a collection of lowers exists, or the intersection of a collection of uppers, then it is also a lower or upper respectively.  \cite{Church 1974a} p.\ 304 notes that it is an open question whether there is a plausible theory with a set with an infinite descending chain but without a membership cycle.}
	More formally:

\begin{theorem} \label{Main Result}
Main Result 
\begin{enumerate}
  \item No lower link intersects any upper link.
  \item 	For any lower \(x\), \(x\) and \(x\+\) constitute a lower ascending link.
  \nopagebreak[4]
  \item For any upper \(x\), \(x\) and \(x\pred \) constitute an upper descending link.
\end{enumerate}
 \end{theorem}  
	 
\begin{proof}
In three parts.
\begin{enumerate}
   \item Corollary \ref{3} states that nothing is both an upper and a lower. 
   \item By Lemma \ref{A}, a lower’s successor is distinct, by Lemma \ref{B} is also a lower, and by Lemma \ref{4} contains it.
   \item By Lemma \ref{C}, an upper’s predecessor is distinct, by Lemma \ref{D} is also an upper, and by Lemma \ref{E} is contained by it.
\end{enumerate}
\end{proof}

 \section{The Russell Paradox Restated}
	 The Russell Paradox, and hence the unrestricted Axiom of Comprehension, imply the claim that one such ascending process, and one such descending process, both \textit{stop,} and at the same set, which is both an upper and a lower.  Stated formally, arranged to display the contrast with the preceding theorem:
	 
	 \begin{theorem} Russell Paradox Restated \label{Russell Paradox Restated} \\
	 Let \(R\) be a strictly Russellian set.
\begin{enumerate}
  \item \(R\) is both an upper and a lower.
  \item 	\(R\+ = R\) 
  \nopagebreak[4]
  \item \(R\pred  = R\)
\end{enumerate}
	 
	  \end{theorem}
	 
	\begin{proof}
\begin{enumerate}
  \item By the definition of strictly Russellian.
  \item By \ref{stoppage}, since \(R\) is an upper.
  \item Also by \ref{stoppage}, since \(R\) is a lower.
\end{enumerate}
\end{proof}

   \section{Philosophical Interpretations}\label{Philosophical Interpretations}

	 Leaving the realm of proof, and entering that of intuition, I claim that the above restatement of the \MyIndex{Russell Paradox} shows that the \MyIndex{Axiom of Comprehension} does not count as a worthwhile intuition deserving of rescue or repair.\footnote{See \textit{contra,} for example, the quotations below from Gödel in §\ref{The Platonist Conception} “The Platonist Conception” and Church in §\ref{Ill-Founded Iterative Conception} “The Ill-Founded Iterative Conception.”}  
	 On close examination, the conflation of predicates and sets (discussed below; “Begriffsumfange oder Klassen” in Frege’s terminology, “\MyIndex{extensions of concepts} or \MyIndex{classes}”) is not essentially correct, with minor blurring, as Gödel would like and Frege feels forced to assume.\footnote{See the quotation below in §\ref{Restrictive Repair Attempts} “Restrictive Repair Attempts.”}  A better visual analogy will be an optical illusion, where the initial impression turns out to be without foundation, and must be adjusted in the light of reality.  
	 
	 As discussed in §\ref{Shared Intuitions} “Shared Intuitions”, a responsible intuition of a generalization must allow intuiting one of its instances, in this case the existence of the \MyIndex{Russell Set}.  
	 This intuition depends on one’s basic intuition about the nature of set existence, which seems to come in three relevant varieties, which I will call the iterative conception, the lasso, and the Platonist.\footnote{I do not claim that my analysis is relevant to some very different conceptions: positive and related theories such as those of Malitz and Barwise \& Moss; ill-founded set theories without a universal set, such as Forti \& Honsell’s and Aczel’s; and the type theory of  \MyIndex{Whitehead} \& \MyIndex{Russell}’s \MyIndex{\textit{Principia Mathematica}}.}  For each of these conceptions, a supposed intuition of the existence of the Russell Set does not withstand scrutiny: the supposed Comprehension intuition does not cohere with the conception.  A fourth category, which I discuss first, is restrictive repair attempts without a conspicuous intuition behind them, for instance as in Frege’s hurried addition (below) to his \textit{Grundgesetze} in reaction to Russell’s letter, or Quine’s New Foundations \cite{Quine 1937}.
	 
 \subsection{Restrictive Repair Attempts} \label{Restrictive Repair Attempts}
 To Frege’s credit, his reaction to Russell’s letter describing his paradox was acceptance rather than denial; he rushed into print, however, a hurried restriction (without a philosophical justification) to his Axiom V  (which in modern terminology is a \MyIndex{higher-order} generalization of the unrestricted \MyIndex{Axiom of Comprehension}) \cite{Frege 1893} “Afterword.”  This attempt is now generally called \MyIndex{Frege’s Way Out}.  

\begin{quotation}

Wie sollen wir uns hierzu stellen?  Sollen wir annehmen, das Gesetz vom ausgeschlossenen Dritten gelte von den Klassen nicht?  Oder sollen wir annehmen, es gebe Fälle, wo einem unanfechtbaren Begriffe keine Klasse entspreche, die sein Umfang wäre?... ... 

Es bleibt also wohl nichts anderes \"ubrig, als die \textbf{Begriffsumfange oder Klassen} als Gegenstände 
im eigentlichen und vollen Sinne 
dieses Wortes anzuerkennen, zugleich
aber einzuräumen, dass die bisherige
Auffassung der Worte "Umfang eines 
Begriffes" einer Berichtigung bedarf.

\textit{What attitude should we take to this?  Should we assume the law of excluded middle fails for classes?  Or should we assume that there are cases where to an incontestable concept no class corresponds as its extension?... }

\textit{So presumably nothing remains but to recognise \textbf{extensions of concepts or classes} as objects in the full and proper sense of the word, but to concede at the same time that the erstwhile understanding of the words “extension of a concept” requires correction. } \cite{Frege 1893} “Afterword,” pp.\ II.254–6, Ebert \& Rossberg translation.  Emphasis added.
\end{quotation}

Approximately four years later he apparently realized that his restricted axiom was inadequate for crucial proofs (\cite{Dummett 1991} p.\ 5).  He abandoned extensions of concepts completely and spent much of the rest of his career in apparent despair; he later started an attempt to found mathematics without using classes \cite{Frege 1924}.  In 1955, Quine showed that the restricted axiom proved that there was at most one object \cite{Quine 1955}.\footnote{Compare also Quine’s remarks on p.\ 156 with the definition of \MyIndex{Lower Russellian Set} and Lemma \ref{B}:  “we might require of a satisfactory reinterpretation that \(\hat{x}(\phi x)\) [the extension of the concept \(\phi \)] continue to contain only objects \(x\) such that \(\phi x\)....   \(\hat{x}(\phi x)\) must not be a subclass of any further class which contains only objects \(x\) such that \(\phi x\)...   It may happen, for some open sentences, that every class whose members all fulfill the sentence is a subclass of a further class whose members all fulfill the sentence, so that there is no final class of the kind.”  See also \cite{Dummett 1991} p.\ 316.}  In 2009 Cook derived a contradiction based on the \MyIndex{Liar Paradox} from it \cite{Cook 2009}.

\smallskip
	 I claim that Frege indulged uncharacteristically here in what I would call the Fallacy of the Only Solution, and a flawed attempted correction (without a philosophical conception), to avoid facing squarely an instance where “to an incontestable concept no class corresponds as its extension.”  While modern set theories coherently treat “classes as objects in the full and proper sense of the word,” we now recognise that “extensions of concepts” and “classes” are not synonyms:  There exist both indefinable sets and extensionless concepts.

\bigskip
In 1937, Quine published his \MyIndex{New Foundations} set theory (\cite{Quine 1937}), which has a syntactically-restricted \MyIndex{Axiom of Comprehension}, also without a philosophical justification.  
In 1940 Quine presented a variant of this, \MyIndex{ML}, which Rosser showed suffered from the \MyIndex{Burali-Forti paradox} (revised edition of \cite{Quine 1951}, p.\ ix).

\bigskip  

From a Formalist perspective, inconsistencies simply remove all interest from Frege’s original system and his Way Out, and from Quine’s ML. 
New Foundations has fared much better for a Formalist; no contradiction has been proved in it, and considerable work has been done towards a relative consistency proof with \MyIndex{Zermelo-Frænkel set theory}.  It has one of Frege’s key desiderata, equinumerosity equivalence classes as sets, though its account of cardinality is counter-intuitive:  The cardinality of the universal set is different from that of the set of all singletons; although both are sets, the intuitively obvious mapping \textendash \  the \MyIndex{singleton function} \textendash \  cannot be a set (\cite{Holmes 1998}, p.\ 110, 131).  New Foundations also disproves the Axiom of Choice (\cite{Specker 1953}), though its variant NFU, which allows urelements, does not, and can be shown consistent in \MyIndex{Zermelo Set Theory} (\cite{Forster 1992} p.\ 67).  

\medskip
For Platonists, among others, \MyIndex{New Foundations} is not a philosophically satisfying solution to the Russell Paradox.  As Sir Michael Dummett wrote (\cite{Dummett 1991} p.\ 230), reflecting both his \MyIndex{Intuitionist} perspective and Frege’s \MyIndex{Logicism}:

\begin{quotation}
	Whatever mathematicians profess, mathematical theories conceived in a wholly formalist spirit are rare.  One such is Quine’s \MyIndex{New Foundations} system of set theory, devised with no model in mind, but on the basis of a hunch that a purely formal restriction on the comprehension axiom would block all contradictions.  The result is not a mathematical theory, but a formal system capable of serving as an \textit{object} of mathematical investigation:  without some conception of what we are talking about, we do not have a theory, because we do not have a subject-matter.
\end{quotation}

From a Platonist perspective such as Gödel’s (below), such concerns (along, perhaps, with mathematical elegance) are primary, and consistency (in sharp contradistinction to Formalism) while a necessary condition, is not alone sufficient.  Thus New Foundations would resemble \MyIndex{Frege’s Way Out} and \MyIndex{ML} in their early years, before their undesirable consequences (including, but not limited to, inconsistency) were discovered.

\subsection{The Iterative Conception}
	 The most influential variety of basic intuition is the iterative conception of set (e.g., \cite{Boolos 1971}); according to this intuition, sets are constructed in a manner understood via a temporal metaphor.  There are two subvarieties:  orthodox well-founded theories such as Zermelo-Fraenkel, and an ill-founded theory by Alonzo Church.
	 
	 For both subvarieties, the intuition implies the Axiom of the Existence of Successor quite directly, and requires the Axiom of the Existence of Predecessor slightly less directly:  Whatever metaphorical iteration led to a given set would also lead to its predecessor.  

\subsubsection{Well-Founded Iterative Conception}
	 Orthodox set theories such as Zermelo-Fraenkel\footnote{See \cite{Holmes 2017} for an argument that, since ZF allows sets to be constructed at an earlier level via quantification over sets constructed at a later level, the theory justified by the iterative conception is actually Zermelo Set Theory with \(\Sigma _{2}\) replacement.} satisfy the Axiom of Foundation and hence have no self-membered sets.  Thus all sets are lowers, and there are no uppers.  
	 
	 \smallskip
	 Comprehension’s implication of the existence of \MyIndex{Russell Set} then would amount to asserting the familiar \MyIndex{Mirimanoff} Paradox \cite{Mirimanoff 1917} of the set of all well-founded sets.  Having the conceived iteration not proceed even a single step beyond this set of all well-founded sets, would flout the iterative nature of the basic intuition.
	 
\subsubsection{Ill-Founded Iterative Conception}\label{Ill-Founded Iterative Conception}
Church’s published Set Theory with a Universal Set (\cite{Church 1974a}, p.\ 303) has both ill-founded sets and a 
 well-founded part which resembles that of the base model of Zermelo-Fraenkel with the \MyIndex{Axiom of Choice} used in his relative consistency proof. 
The axiom schemas of \textit{Aussonderung} (Separation), Replacement, and a sequence generalizing \MyIndex{Frege-Russell cardinals} are restricted to well-founded sets; so is the Axiom of Power Set, though this restriction (as well as some of the axioms and schemas) is removed in the related theory of \cite{Mitchell 1976}.
Church presents no intuition behind his theory here or in his unpublished notes and later related theories (\cite{Sheridan 2016} pp.\ 9–10).
Instead he suggests an historical heuristic:
 \begin{quotation}
  principle of specialized comprehension...  we seek axioms which are special cases of the general comprehension axiom and which promise to maintain consistency while at the same time being adequate for a large variety of mathematical purposes....  may be roughly described by saying that we accept all of naive set theory except the comprehension axiom, and then assume as many special cases of the comprehension axiom as we dare  (p.\ 297).
\end{quotation}

His published article concludes by suggesting “the search for suitable axioms to be added to the basic axioms,”
with “one source of axioms to be studied for their consistency with the basic axioms” \cite{Hailperin 1944}’s finite axiomatization of \MyIndex{New Foundations} (pp.\ 307–8).
	\smallskip
	
	Assume that we intuit the existence of the \MyIndex{Universal Set} (or, indeed, any upper) while holding some sort of iterative conception of set existence.
	If we also intuit Comprehension, and hence the existence of \MyIndex{Russell Set}, 
	then one more iteration would yield its successor, which by the ascending link part of Theorem \ref{Main Result} is distinct.  The iterative process would also yield the \MyIndex{Russell Set}’s predecessor, which is likewise distinct.
	
	The iterative intuition requires continuing both these processes.  Calling a halt would be completely arbitrary; halting the successor iteration is analogous to claiming that the successor function on the natural numbers stops at some particular integer.  The predecessor iteration generates something new, given any upper: an indefinitely-extensible process of descending links.
	
	The analogous descending image I suggest for this is a non-standard model of Peano Arithmetic, visualizing the descending links as decreasing non-standard (i.e., externally non-finite) positive integers.  Visualize the ascending links as before: increasing non-negative standard (i.e., externally finite) integers.  Intuiting the existence of the Russell Set then is analogous to claiming that both processes stop, and at the same place.  This stopping point is then analogous to something which is both the last standard integer and the first non-standard integer, and which has neither a successor nor a predecessor.  Thus the consequences of the combined intuition do not cohere with the basic intuition.

	 \subsection{The Lasso and the Wand}
	 I address here the intuition presented in \cite{Forster 2008} on ill-founded set theories; he attributes the intuition of the lasso to \cite{Boolos 1971} on well-founded sets, who in turn credits Kripke for the original insight.  Boolos’ use of the lasso intuition, however, is for his iterative conception of well-founded sets, which is discussed above.  Forster’s employment of the lasso plus a wand for ill-founded sets is different; for generality and simplicity, I will use only the basic part of the intuition here. 
	 
	 This intuition also relies on a temporal metaphor, where some preexisting sets are selected by a lasso, and the wand creates two new sets: one containing the lassoed sets, and the other containing those preexisting sets not lassoed.  Note that the intuition must also determine whether each new set contains itself.  (Forster’s further intuitions supply these answers, and determine whether the new sets will contain each other, as well as later sets when this process is iterated; but that is not needed for our purposes.)  
	 
	 \smallskip
	 The crucial observation is simply that, whichever decision is made as to self-membership, and no matter what the lasso has collected, the decision avoids the Russell Paradox:  If a new set is self-membered, it is not a lower by Lemma \ref{1}; if it is not self-membered, it is not an upper by Lemma \ref{2}.  Either way, the intuition that some such set is strictly Russellian amounts to a refusal to make the required decision about membership in the newly-created set.

\bigskip	 
	\cite{Button 2023} contains a more general intuition, with wands satisfying various requirements, but not lassos, and with a modified temporal metaphor:  Sets are viewed, not as constructed at various levels, but merely indicated there.
	
	\smallskip
	The generality shifts the avoidance of paradox to requirements on particular wands.  The wand for equinumerosity (and for Church’s generalization), for example, satisfies these requirements (§11).  The obvious path to the \MyIndex{Russell Set}, the equivalence relation of being equally self-membered (i.e., \(x \in  x \iff  y \in  y\)) is blocked by the requirements on a wand/set theory in his definition 5.1.
	 
\subsection{The Platonist Conception}\label{The Platonist Conception}
Mathematical Platonism is now less popular than the iterative conception, but it is worth recalling that it was part of the outlook of two of the greatest logicians in history, and it may still be relevant to understanding the \MyIndex{Russell Paradox}.
 \begin{quotation}
Die Zeit ist nur ein psychologisches Erforderniss zum Zählen, hat aber mit dem Begriffe der Zahl nichts zu thun.

\textit{Time is only a psychological necessity for numbering, it has nothing to do with the concept of number.}  

\cite{Frege 1884} §40 p.\ 53, J.L. Austin’s translation.
\end{quotation}

	 \begin{quotation}
 ... it seems that the vicious circle principle in its first form [“definable only in terms of,” preceding page] applies only if the entities involved are constructed by ourselves.... 

Classes and concepts may, however, also be conceived as real objects, ...  existing independently of our definitions and constructions.

It seems to me that the assumption of such objects is quite as legitimate as the assumption of physical bodies and there is quite as much reason to believe in their existence.

 \medskip
 
 ... our logical intuitions would then remain correct up to certain minor corrections, i.e., they could then be considered to give an essentially correct, only somewhat “blurred,” picture of the real state of affairs. 
 
 \cite{Gödel 1944}, pp.\ 456, 466–7.
\end{quotation}

Frege’s original conception retains interest for Platonists even after the paradox.  It is of course a profound mystery how it achieved so much while its most significant axiom led to such a short contradiction.  But that achievement is enormous:  Frege provided an absolutely rigorous foundation for arithmetic, and was well on his way to providing a foundation for analysis.  Even \MyIndex{Whitehead} \& \MyIndex{Russell}’s \MyIndex{\textit{Principia Mathematica}} (according to \cite{Gödel 1944} p.\ 448) was substantially less rigorous.  Still less rigorous are typical logic articles today, not to mention proofs (more properly proof sketches) where there is serious disagreement over whether they actually constitute a proof.  Compare Frege’s boast:  
 \begin{quotation}
  Wenn etwa jemand etwas fehlerhaft finden sollte, muss er genau angeben können, wo der Fehler seiner Meinung nach steckt: in den Grundgesetzen, in den Definitionen, in den Regeln oder ihrer Anwendung an einer bestimmten Stelle. 
\end{quotation}
\begin{quotation}
   \textit{If anyone should believe that there is some fault, then he must be able to state precisely where, in his view, the error lies: with the basic laws, with the definitions, or with the rules or a specific application of them.}  \cite{Frege 1893}, p.\ VII, Ebert \& Rossberg translation
\end{quotation}

\medskip	 
	 By the Platonist conception, sets are not constructed in time, but exist independently of our constructions:  all we can do is point to them.  (“\MyIndex{Platonist}” is not ideal terminology; “\MyIndex{Pythagorean}” or “\MyIndex{atemporal}” might do as well.)  A set existence axiom merely points out a preexisting set with the required property; it no more creates a new set than we can create a new Boolean value.  
	 
	 \smallskip
	 Presumably these sets are closed under some interesting operations, but for this analysis, I assume only the usual trivial successor and predecessor axioms of \MyIndex{Incomprehensive Set Theory}.	 
	 In the case of the Russell Set, Theorem \ref{Main Result} makes it clear that no set we point to will be strictly Russellian:  Either the ascending or descending link part of the theorem will demonstrate the inadequacy of any plausible candidate, by pointing to a less bad approximation, which in turn is not strictly Russellian either.  
	 
	 By the conception of sets as existing independently of our constructions, the claim that there is a Russell Set is not, as Gödel would have it, an essentially correct intuition requiring minor correction.  
	 It is simply wishful thinking, analogous to wanting the phrase “sixth Platonic solid” to have a referent, while not actually intuiting which shape this might be.

\subsection{Conclusion}

I claim that Frege was right to abandon sets as extensions of concepts; a sufficiently deep and Platonist contemplation of this distinction, and of his rejection (quoted above) of the temporal metaphor, might have saved him from paradox and despair.  I also argue that this would not require abandoning sets completely, not even equivalence classes under equinumerosity as a formalization of cardinality.  
Rigorous definition by abstraction as the origin of cardinal numbers was the key to Frege’s philosophical anti-formalist doctrine that, as [Soames 2019] p.\ 97 summarizes, “\textit{Numbers are whatever they have to be in order to explain our knowledge of them.}  To discover what they are, we must give definitions of numbers that allow us to deduce what we pre-theoretically know,” 
which is a partial answer to the problem, still with no complete solution, of our knowledge of mathematical truths.  

It is also an alternative to arbitrariness in the implementation of cardinal numbers as initial Zermelo or von Neumann ordinals (\cite{Benacerraf 1965}, p.\ 281), and to the inconsistency of a “univocal role” for Frege cardinals (p.\ 284).\footnote{Frege’s 
definition of numbers in terms of one-one correspondence also pre-emptively rebuts Benacerraf’s claim (p.\ 290) that natural numbers have meaning only in the context of the sequence of all numerals.  A native speaker of a language without an indefinitely extensible numeric notation (\cite{Comrie 2022} p.\ 3), 
but who grasped the notions of bijection and pairs, might also be able to grasp Frege’s notion of the number two without first having to accept a sequence of all numbers.  

Recall Frege’s note (\cite{Frege 1884} §73 p.\ 85) that the etymology is deceptive here:  “\MyIndex{equinumerosity}” morphologically presupposes numerals, and Frege’s term “\MyIndex{gleichzahl}ig” morphologically presupposes numbers.  Austin’s translation “equal” avoids this difficulty but has its own problems.  “eindeutigen Zuordnung” (§62 p.\ 74) is the base concept, which Austin translates as “one-one correlation.”  The morphology may be helpful here:  in English, “one-point(-adjective)(-verb) to-arrange(-noun).”  Even the use of “one” increases the ontology unnecessarily; “each” would do as well.}  
It is an incomplete explanation of the otherwise unreasonable effectiveness of mathematics, since it does not cover higher mathematics, especially fields invented with no thought of practical applications which later acquired them.
\medskip

Considerable technical challenges remain, however, and even before Frege learned of the paradox, avoiding arbitrariness in the definition of real numbers (not to mention that of Julius \MyIndex{Cæsar}, \cite{Dummett 1991} pp.\ 157, 165, 280–1) had no obvious solution.\footnote{Given 
Frege’s suspicion of the unclarity of natural languages, the notorious Julius \MyIndex{Cæsar} problem would not have been relevant to his formalization program until it had reached formal definitions of names of biological entities.  While much of recent philosophy of mathematics and its notation is concerned with human language, this has not been a universal preoccupation.  (It would not be straightforward to read Frege’s two-dimensional concept writing aloud in any human language.  It is also worth recalling that Frege did not confine himself to a \MyIndex{first-order language}.)  

The first writing had no necessary connection with any particular natural language.  (Its inventors probably spoke primarily Sumerian, but the nature of the writing makes certainty impossible  \cite{Michalowski 2004}, pp.\ 20 \& 25.)  Given the current difficulties caused for artificial intelligences by reliance on predominantly one natural language (English), in the long run, the connection between human natural languages and logical writing may turn out to have been temporary.}

\cite{Church 1974a} has an axiom of the existence of a generalization of equinumerosity equivalence classes (including “relation number[s] in the sense of \textit{\MyIndex{Principia Mathematica}},” p.\ 303), though the axiom is restricted to well-founded sets, and considerably more work would need to be done to see whether its cardinalities are even more counterintuitive than Quine’s.  
Church apparently abandoned his unpublished attempts (\cite{Sheridan 2016} Appendix §C.2 pp.\ 99–100) at a unification between his set theories with a universal set and Quine’s set theory.  \censor{\cite{Sheridan 2016}} is a largely unsuccessful attempt to extend Church’s published theory while avoiding one of the difficulties of New Foundation’s cardinalities, the non-existence of the \MyIndex{singleton function}.
\cite{Oberschelp 1973} contains a promising approach for model-building and has the \MyIndex{singleton function} as a set, but both its philosophical underpinnings and the formalism of its main proof are far from obvious.  \cite{Burgess 2005} presents a quite different formalization of cardinal numbers, which might have reduced Frege’s despair over his foundational program.
\bigskip

	 That our intuitions are sometimes correct remains the most profound mystery of mathematics.  The failure of some of our intuitions can be considerably less profound, and needs to be accepted rather than resisted.	 In retrospect, it should have been no more surprising that not all predicates correspond to a set, than that not all sets correspond to a definable predicate.
	 Attempting to repair such a broken intuition would be the metaphorical equivalent, not of correcting blurred vision, but of searching for a smaller version of something after it turned out to be an optical illusion.

  \section{Appendix:  A Less Pathological Counterexample to a Self-Membership Query of Forster}
  
  In email about an earlier version of this paper, Thomas Forster raised the question of whether (in my notation): \(x\in x \implies  x\pred  \in  x\pred \).  There are simple pathological counterexamples, such as a \MyIndex{Quine atom} (whose existence seems independent of \MyIndex{New Foundations}) \(q\) such that \(q = \{q\}\).  (Observe that \(q \in  q\),  and \(q\pred  = q \setminus  \{q\} = \{q\} \setminus  \{q\} = \emptyset \), and \(\emptyset  \notin   \emptyset \).)
  I am grateful to an anonymous referee for the observation that, in the context of \MyIndex{New Foundations}, a counterexample may be obtained with a \MyIndex{Rieger-Scott permutation}.

The generalization is provable in \MyIndex{Incomprehensive Set Theory} of any upper, by 4.4 \& 3.2, and trivially true of any lower by 3.1.  It holds in some simple cases in constructions like Church’s and \censor{mine}, e.g., the Universal Set U (where \(U\in U \wedge  U\pred  \in  U\pred \)), and equivalences classes (or simple combinations thereof) such as the set of all singletons, which is not a member of itself.

That is however a limitation of those constructions rather than a general trait, according to the shared intuitions in §\ref{Shared Intuitions} of sets as extensional objects rather than only extensions of concepts.  It is a relic of the goal of this technique, to provide extensions for some broad concepts like equinumerosity equivalence classes; the constructions do not go very far beyond that goal.  By this intuition, whether one object belongs to a given set should be independent, in normal cases, of whether some other object belongs to it.  (A partial formulation of this intuition would be generalizations of the Axioms of Predecessor and Successor, to the existence of \(x \setminus  \{y\}\) and of \(x \cup  \{y\}\), even when \(x \neq  y\).)

A non-pathological counterexample (which nonetheless exposes issues about individuation and identity), \(N\), with \(N = U \setminus  \{N\pred \}\), where \(U\) is the \MyIndex{Universal Set}, may be obtained via a small modification to the construction in \cite{Sheridan 2016}.  (This technique apparently would not work with the similar construction in \cite{Church 1974a}.)

This construction can be repeated, so that there is another such set with \(P = U \setminus  \{P\pred \}\), but this raises the question of whether it would accord with the intuition of sets as extensional objects to have two such sets.
Arguably this intuition requires that there be only one such object, but it is not clear what  general principle would justify this.  \MyIndex{Finsler}’s Axiom of \MyIndex{Strong Extensionality} might suffice, but may present difficulties \cite{Forster 1998}.
Such questions of identity, and axiomatising intuitions about the nature of sets as extensional objects, are promising areas for future work.

I here sketch the construction referenced in footnote \ref{footnote_Forster}, of a set satisfying \(N = U \setminus  \{N\pred \}\); consequently \(N\in N\) but \(N\pred \notin N\pred \).  For non-readers of \cite{Sheridan 2016}, I begin with an informal summary of the construction technique.

\subsection[Brief Summary of  Membership Relation in \cite{Sheridan 2016}]{Brief Summary of  \(\in _3\) in \cite{Sheridan 2016}}

The \MyIndex{base theory} is \MyIndex{RZFU}, a theory similar to Zermelo-Frænkel with a global well-ordering, with the Axiom of Extensionality modified to permit urelements, and a bookkeeping axiom to keep track of them.  (Other differences from conventional ZFC with urelements need not concern us here,\footnote{E.g., the theory does not assume the Axiom of \MyIndex{Foundation}, but the axioms of \MyIndex{Replacement}, \MyIndex{Power Set}, and \MyIndex{Sum Set} are restricted to well-founded sets} \censor{and were in retrospect not worth the effort}.)
The construction below will use a relation similar to the main membership relation, \(\in _3\), in \cite{Sheridan 2016}, which I will call \(\in _4\).  
For the benefit of non-readers of that work, I will omit many minor details and sketch only the relevant parts of the construction.  A somewhat less brief summary is in \cite{Sheridan 2014}.

The \(\in _3\) construction is an interpretation over \censor{my} base theory (whose membership relationship is called \(\in _0\) for clarity), where \(\in _3\) extends \(\in _0\) when an urelement is on the right.
The point to the construction is to define \(\in _3\) in terms of \(\in _0\), and then prove the \Def{interpretation} of each axiom, i.e., its definitional expansion, including expansion of symbols defined in terms of \(\in _3\).  This then provides a purely syntactic\footnote{Despite Church’s use of the word “model,” e.g., page 305 of \cite{Church 1974a}.} proof of the consistency of the theory relative to the base theory:  Any proof of a contradiction from the axioms of interest can be converted into a proof of a contradiction in the base theory.

\subsubsection{The Sequence of Equivalence Relations}
Much of the complexity in both Church’s and  \censor{my} relative consistency proofs came from \censor{our} sequences of equivalence relations, which are not relevant to the current construction, and which I will ignore except for two trivial ones, the first and the last in the sequence.  \censor{My} sequence was from \(0\) to an arbitrary ordinal \(\mu \), which was a generalization of Church’s choice of \(1\) through an arbitrary natural number \(m\).  For greater notational uniformity in that construction, \censor{I} supplemented Church’s relations with two trivial ones, \(\Bumpeq _0\) and \(\Bumpeq _\mu \), which will be the only relations used here: \(\Bumpeq _0\) is the trivial relation true of any two objects, and \(\Bumpeq _\mu \) is equality.  (The other ordinals between \(0\) and \(\mu \), if any, will be elided here with “... ”.) 

\censor{My} construction relied on a form of Global Choice to select a representative from each equivalence class; given the restriction to two trivial relations here, the user may think of \(\mu \) as \(1\), and of the representative for \(x\)’s \(\Bumpeq _\mu \)-equivalence class (denoted “\(\mu \rep(x)\)”) as \(x\) itself.  The representative of the sole \(\Bumpeq _0\)-equivalence class (denoted “\(0\rep(x)\)”, for any \(x\)) may be thought of as \(\emptyset \) for the sake of definiteness.

Sequences of length \(\mu +1\) (which in conformity to Church’s terminology  \censor{I} called \(\mu +1\)-tuples and surround with “\((... )\)” rather than “\(<... >\)”) are used to tag urelements.  These urelements represented new sets, such as the Universal Set, in the sense of \(\in _3\).  (\(\in _3\) coincides with \(\in _0\) when a non-urelement is on the right side.)  The tagging was via a bijective formula \(\Upsilon ''(x)\) (locally abbreviated, for convenience, “\(*\)”, with “\((... )\)” omitted; see \censor{§20.2 “Excess Urelements”} of \cite{Sheridan 2016} for details.)  The \(\mu +1\)-tuple tagging an urelement was called its \textbf{Index}.

A simplifying intermediary \censor{I} added to Church’s definition was what \censor{I} called the “sprig of \(x\) for an urelement \(u\) with tag \(L\)”:  the collection of pairs \(<j, j\rep(x)>\) (for \(j \leq  \mu \)) such that \(j\rep(x)\) was in the \(j\)\textsuperscript{th} element of the tag \(L\) of the given urelement \(u\).

\subsubsection{\(\in _3\)}
The membership relation \(\in _3\) was extended for urelements (in the sense of \(\in _0\)), by stipulating for an urelement \(u\), that \(x \in _3 u\) if the sprig of \(x\) for \(L\) had an odd number of members.  In the case we consider here, considering only ordinals \(0\) and \(\mu \), the only possible odd number is one; i.e., if the sprig of \(x\) for \(L\) contains \(0\rep(x)\) or \(\mu \rep(x)\) but not both.

Informally, in this type of construction, self-membership is mainly signalled only by the zeroth (complement) component, and can otherwise normally be affected only by the last (\(\mu \)) component listing exceptions.  Components signalling Church’s or  \censor{my} equivalence relations don’t seem to provide self-membership; for instance, the Frege-Russell cardinal \(1\) (the set of all singletons), has more than one element.  (This type of interpretation doesn’t provide generalized Frege-Russell cardinals for big sets.) 

Thus, for instance, \(*(\{0\rep(0)\}, ...  \{\})\) will represent the \MyIndex{Universal Set}, U, in the sense of \(\in _3\):  The \(0\)\rep\  of anything is in the zeroth element of the sequence, and the other element(s) of the sequence are empty; thus the sprig of anything for tag \((\{0\rep(0)\}, ...  \{\})\) has one member, and one is odd.  Similarly, \(U\pred \) would be represented by \(*(\{0\rep(0)\}, ...  \{\mu \rep(U)\})\).

\subsection{\(\in _4\)}

I will here briefly sketch the small changes needed for a new membership relation \(\in _4\), which will provide the required counterexample.  The only change from \(\in _3\) in §20.3 of \cite{Sheridan 2016} is a tweak to the bijective formula \(\Upsilon ''\):  

Choose two distinct urelements \(M\) and \(N\).  Consider the following two indexes:

\[ n = (\{0\rep(0)\}, ...  \{\mu \rep(M)\}) \]

\[ m = (\{0\rep(0)\}, ...  \{\mu \rep(M), \mu \rep(N)\}) \]

We tinker with the bijective formula \(\Upsilon ''\) mapping Indexes to urelements.  Define \(\Upsilon '''\) as \(\Upsilon ''\), except that we swap values so that it maps (if it does not already) \(n\) to \(N\) and \(m\) to \(M\), and (to avoid collisions) \(\Upsilon ''\text{-}inverse(N)\) to \(\Upsilon ''(n)\), and \(\Upsilon ''\text{-}inverse(M)\) to \(\Upsilon ''(m)\).  By my customary abuse of notation, I will again recycle the symbol “\(*\)” and locally abbreviate \(\Upsilon '''(x)\) to \(*x\), without parentheses.

The definition of \(\in _4\) is then textually identical to that for \(\in _3\), given the new local definition of \(*\).

\subsection{Properties of M and N}

Thus we have 

\[ N = *(\{0\rep(0)\}, ...  \{\mu \rep(M)\}) \]
\[ M = *(\{0\rep(0)\}, ...  \{\mu \rep(M), \mu \rep(N)\}) \]

\begin{lemma} \(N\) is, in the sense of \(\in _4\), \(U \setminus  \{M\}\) \end{lemma}

\begin{proof}
	By definition, \( x \in _4 N \iff  odd(sprig(N, x)) \).
There is only one \(0\rep\), so \(0\rep(x)\) is always in \{0\rep(0)\}.
So the sprig has an odd number of members unless \(\mu \rep(x) \in  \{\mu \rep(M)\}\) is also true, i.e., iff \(x = M\).  So \(x \in _4 N \iff  x \neq  M\), as required.\end{proof}

\begin{corollary} \(N \in _4 N\) \end{corollary}

\begin{lemma}  \(M\) is, in the sense of \(\in _4\), \(U \setminus  \{M, N\}\) \end{lemma}

\begin{proof}
	Similarly, \(x \in _4 M \iff  odd(sprig(M, x))\),
which is true unless \(\mu \rep(x) \in  \{\mu \rep(M), \mu \rep(N)\}\), iff
\(x = M \lor  x = N\).

So \(x \in _4 M \iff  x \neq  M \wedge  x \neq  N\), as required.\end{proof}

\begin{corollary} \(M \notin _4 M\) \end{corollary}

\begin{corollary} \(M\) is, in the sense of \(\in _4\), \(N\pred \) \end{corollary}

\begin{corollary} In the sense of \(\in _4\), \(N \in _4 N \wedge  N\pred  \notin _4 N\pred \) \end{corollary}
\textit{Q.E.D.}

\section*{Acknowledgments}
	 \censor{This paper began in 2005 as a reconstruction of my essay on the \MyIndex{Russell Paradox} for \MyIndex{Angus Macintyre}’s Philosophy of Mathematics course in 1980, which provided the philosophical motivation for my work on Church’s Set Theory with a Universal Set \cite{Church 1974a}.  The reconstruction was heavily influenced by my unpublished article on singletons in Skala’s set theory \cite{Skala 1974}, which had been implicitly superseded by K\"uhnrich and Schultz’s earlier article \cite{Kuhnrich 1980}; Skala’s terminology is the origin of my terms “upper” and “lower.”  Various of Thomas Forster’s articles and corrections were also an influence, though the reconstruction was nearly complete before I read the most relevant of these, the November 2005 version of “The Iterative Conception of Set” \cite{Forster 2008}.}
	 
	 \censor{I am grateful to two anonymous referees for criticisms, suggestions, and demands for additional philosophical arguments, to my wife Olga Miroshnychenko for assistance while writing and rewriting, to Roger W. Janeway for discussions of the orginal version of this article and corrections and a metaphor for the final version, and to Murdoch J. Gabbay for suggestions which led to substantial improvements and additions.}

\fi 

\end{document}